\documentclass[a4paper,french,titlepage,12pt]{article}%
\usepackage{amssymb}
\usepackage{amsmath}
\usepackage{amsfonts}
\usepackage{sw20nhj}
\usepackage{euscript}
\usepackage{doublespace}
\usepackage{theorem}
\usepackage{color}
\usepackage{colortbl}
\usepackage{graphicx}%
\providecommand{\U}[1]{\protect\rule{.1in}{.1in}}
\newtheorem{theorem}{Th\'{e}or\`{e}me}

\newtheorem{corollary}[theorem]{Corollaire}

\newtheorem{lemma}[theorem]{Lemme}

\newtheorem{proposition}[theorem]{Proposition}
\theorembodyfont{\rmfamily}

\newenvironment{proof}[1][Preuve]{\textbf{#1.} }{\ \rule{0.5em}{0.5em}}
\def\QTR#1{\csname #1\endcsname}
\begin{document}

\title{Probl\`{e}me de Plateau complexe feuillet\'{e}. Ph\'{e}nom\`{e}nes de
Hartogs-Severi et Bochner pour des feuilletages CR singuliers.}
\author{Henkin G., Michel V.\vspace{-15cm}\vspace{-15cm}}
\maketitle
\date{19 septembre 2011}

\begin{abstract}
The purpose of this paper is to generalise in a geometric setting theorems of
Severi, Brown and Bochner about analytic continuation of real analytic
functions which are holomorphic or harmonic with respect to one of its
variables. We prove in particular that if $N$ is a real analytic levi-flat
annulus in an open set of $\mathbb{R}^{n}\times\mathbb{C}^{2}$, then one can
find $\mathcal{X}\subset\mathbb{R}^{n}\times\mathbb{C}^{2}$ such that
$\mathcal{X}\cup N$ is a levi-flat real analytic subset and $\mathcal{X}$
fills $N$ in the sense that the boundary of the integration current of
$\mathcal{X}$ is a prescribed smooth submanifold of $N$ foliated by real
curves. Moreover, real analytic functions on $N$ whose restrictions to complex
leaves are harmonic extend to $\mathcal{X}$ in functions of the same kind. We
give also a theorem when the prescribed boundary is a cycle.\medskip\medskip

Le but de cet article est de g\'{e}n\'{e}raliser g\'{e}om\'{e}triquement des
th\'{e}or\`{e}mes de Severi, Brown et Bochner portant sur le prolongement
analytique des fonctions r\'{e}elles analytiques qui sont holomorphes par
rapport \`{a} l'une de leurs variables. Nous \'{e}tablissons en particulier
que si $N$ est un anneau l\'{e}vi-plat r\'{e}el analytique d'un ouvert de
$\mathbb{R}^{n}\times\mathbb{C}^{2}$, il est possible de trouver
$\mathcal{X}\subset\mathbb{R}^{n}\times\mathbb{C}^{2}$ tel que $\mathcal{X}%
\cup N$ soit un sous-ensemble r\'{e}el analytique qui remplisse $N$ au sens
o\`{u} le bord du courant d'int\'{e}gration port\'{e} par $\mathcal{X}$ est
une sous-vari\'{e}t\'{e} lisse de $N$ feuillet\'{e}e par des courbes
r\'{e}elles. En outre, les fonctions r\'{e}elles analytiques sur $N$ dont les
restrictions aux feuilles complexes sont harmoniques se prolongent \`{a}
$\mathcal{X}$ en fonction du m\^{e}me type. Nous donnons aussi un
th\'{e}or\`{e}me quand le bord prescrit est un cycle.

\end{abstract}

\noindent G. Henkin : Institut Math\'{e}matiques de Jussieu, 75252 Paris Cedex
; henkin@math.jussieu.fr

\noindent V. Michel : Institut Math\'{e}matiques de Jussieu, 75252 Paris Cedex
; michel@math.jussieu.fr

\noindent\textbf{Mots cl\'{e}s} : probl\`{e}me de Plateau complexe
feuillet\'{e}, ph\'{e}nom\`{e}nes de Severi-Brown-Bochner, fonction de Green

\noindent\textbf{Classification} : 32D15, 32C25, 32V15, 35R30, 58J32%

\begin{spacing}{1.1}%

\section{Enonc\'{e}s}

Lorsque $\Omega$ est un ouvert born\'{e} de $\mathbb{C}$ dont le bord est
lisse et connexe, on sait depuis les travaux de Sohotsky en 1873 qu'une
fonction $u$ d\'{e}finie et continue sur le bord $\gamma$ de $\Omega$ se
prolonge holomorphiquement \`{a} $\Omega$ si et seulement si elle v\'{e}rifie
la condition des moments, \`{a} savoir que%
\[
\int_{\gamma}hudz=0
\]
pour toute fonction enti\`{e}re $h$. Une version g\'{e}om\'{e}trique et plus
g\'{e}n\'{e}rale de ce r\'{e}sultat est \'{e}tablie dans~\cite{WeJ1958c},
\cite[th. 4.2]{HeG1995} et \cite[th. 4]{DiT1998b}~: \'{e}tant donn\'{e} un
couple $\left(  \gamma,\mu\right)  $ o\`{u} $\gamma$ est une courbe r\'{e}elle
orient\'{e}e connexe de $\mathbb{C}^{n}$ satisfaisant une hypoth\`{e}se de
r\'{e}gularit\'{e} tr\`{e}s faible et $\mu$ une mesure sur $\gamma$,
l'annulation de tous les moments de $\mu$ sur $\gamma$, c'est \`{a} dire%
\begin{equation}
\forall h\in\mathcal{O}\left(  \mathbb{C}^{n}\right)  ,~\int_{\gamma}hd\mu=0
\label{F/ Mom mes sur courbe}%
\end{equation}
caract\'{e}rise l'existence d'une courbe complexe$^{(}$\footnote{Une courbe
complexe est dans cet article un ensemble (complexe) analytique de dimension
$1$ et une surface de Riemannn est une courbe complexe lisse et connexe.}%
$^{)}$ $\mathcal{X}$ de volume fini dont $\gamma$ est le bord au sens des
courants et d'une $\left(  1,0\right)  $-forme holomorphe $\varphi$ sur
$\mathcal{X}$ dont $\mu$ est la valeur au bord, c'est \`{a} dire v\'{e}rifie
au sens des courants $d\left(  \left[  \mathcal{X}\right]  \wedge
\varphi\right)  =\mu$, $\left[  \mathcal{X}\right]  $ \'{e}tant le courant
d'int\'{e}gration sur $\mathcal{X}$ et $\mu$ \'{e}tant identifi\'{e}e \`{a}
son extension simple \`{a} $\mathbb{C}^{n}$. Pr\'{e}cisons que $\mathcal{X}$
n'\'{e}tant pas forc\'{e}ment lisse, une $\left(  1,0\right)  $-forme
(faiblement) holomorphe $\varphi$ sur $\mathcal{X}$ est localement la
restriction \`{a} $\mathcal{X}$ d'une $\left(  1,0\right)  $-forme
m\'{e}romorphe $\widetilde{\varphi}$ telle que $\overline{\partial}\left(
\left[  \mathcal{X}\right]  \wedge\widetilde{\varphi}\right)  =0$.

Il s'av\`{e}re que la condition des moments gouverne aussi la possibilit\'{e}
de r\'{e}aliser un triplet $\left(  \gamma,u,\alpha\right)  $ o\`{u} $u\in
C^{1}\left(  \gamma\right)  $ et $\alpha\in C^{1}\left(  \gamma,\Lambda
^{1,0}T^{\ast}\mathbb{C}^{n}\right)  $ comme une donn\'{e}e de Cauchy du
Laplacien d'une courbe complexe $\mathcal{X}$ \`{a} construire que $\gamma$
borderait. Lorsque $\mathcal{X}$ est une surface de Riemann ouverte abstraite
bord\'{e}e par $\gamma$ au sens des vari\'{e}t\'{e}s \`{a} bord, une telle
donn\'{e}e est un triplet $\left(  \gamma,u,\alpha\right)  $ comme ci-dessus
avec la condition que $\alpha=\partial\widetilde{u}\left\vert _{\gamma
}\right.  $ o\`{u} $\widetilde{u}$ est le prolongement harmonique de $u$ \`{a}
$\mathcal{X}$. Dans cette situation, $\alpha$ peut \^{e}tre exprim\'{e}e \`{a}
partir de l'op\'{e}rateur de Dirichlet-\`{a}-Neumann $N_{\mathcal{X}}$ de
$\mathcal{X}$, c'est \`{a} dire de l'op\'{e}rateur qui \`{a} $u\in
C^{1}\left(  \gamma,\mathbb{R}\right)  $ associe $\frac{\partial\widetilde{u}%
}{\partial\nu}\left\vert _{\gamma}\right.  $ o\`{u} $\widetilde{u}$ est le
prolongement harmonique de $u$ \`{a} $\mathcal{X}$ et $\nu$ est le champ de
vecteurs unitaires qui dirigent le long de $\gamma$ la normale ext\'{e}rieure
\`{a} $\mathcal{X}$. En effet, si on d\'{e}signe par $\tau$ le champ des
vecteurs tangents \`{a} $\gamma$ tel qu'en chaque point $x$ de $\gamma$,
$\left(  \nu_{x},\tau_{x}\right)  $ soit une base orthonorm\'{e}e directe de
$T_{x}\mathcal{X}$ , si on note $\left(  \nu^{\ast},\tau^{\ast}\right)  $ la
duale de $\left(  \nu,\tau\right)  $ et si pour $u\in C^{1}\left(
\gamma\right)  $ on pose
\[
L_{\mathcal{X}}u=\frac{1}{2}\left(  N_{\mathcal{X}}u-iTu\right)
\]
o\`{u} $T$ d\'{e}signe la d\'{e}rivation par rapport \`{a} $\tau$, alors
\[
\forall u\in C^{1}\left(  \gamma\right)  ,~\partial\widetilde{u}\left\vert
_{\gamma}\right.  =\left(  L_{\mathcal{X}}u\right)  \left(  \nu^{\ast}%
+i\tau^{\ast}\right)  .
\]
Pour qu'un triplet $\left(  \gamma,u,\alpha\right)  $ puisse \^{e}tre une
donn\'{e}e de Cauchy, $\alpha$ doit donc \^{e}tre de la forme $\frac{1}%
{2}\left(  u^{\prime}-iTu\right)  \left(  \nu^{\ast}+i\tau^{\ast}\right)  $
o\`{u} $u^{\prime}\in C^{0}\left(  \gamma\right)  $, condition qui garde un
sens m\^{e}me lorsque l'existence m\^{e}me de $\mathcal{X}$ n'est pas connue
pour peu que $\gamma$ soit plong\'{e}e dans une vari\'{e}t\'{e} complexe car
$\nu^{\ast}$ peut \^{e}tre alors d\'{e}fini a priori de fa\c{c}on naturelle.
Une donn\'{e}e de Cauchy \'{e}tant ainsi pos\'{e}e dans $\mathbb{C}^{n}$, le
probl\`{e}me de construire \`{a} la fois $\mathcal{X}$ et un prolongement
harmonique \`{a} la fonction donn\'{e}e sur $\gamma$ s'inscrit dans le
probl\`{e}me de reconstruire une surface de Riemann \`{a} partir de son de son
op\'{e}rateur de Dirichlet-\`{a}-Neumann. Dans le cas g\'{e}n\'{e}ral, il est
naturel qu'une courbe complexe $\mathcal{X}$ bord\'{e}e par une courbe
r\'{e}elle pr\'{e}sente des singularit\'{e}s et la notion de fonction
harmonique doit \^{e}tre \'{e}tendue pour que le probl\`{e}me
pr\'{e}c\'{e}dent garde un sens.

Par d\'{e}finition, une fonction harmonique sur une courbe complexe
$\mathcal{X}$ est une fonction $u$ qui est harmonique au sens usuel sur la
partie r\'{e}guli\`{e}re $\operatorname*{Reg}\mathcal{X}$ de $\mathcal{X}$ et
faiblement harmonique sur $\mathcal{X}$, ce qui signifie que $\partial u$ est
faiblement holomorphe. Une fonction $u$ est dite harmonique multivalu\'{e}e
sur $\mathcal{X}$ si elle admet le long de tout chemin trac\'{e} dans la
partie r\'{e}guli\`{e}re de $\overline{\mathcal{X}}$ une d\'{e}termination en
fonction harmonique usuelle et si $\partial u$ est une $\left(  1,0\right)
$-forme faiblement holomorphe (univalu\'{e}e) sur $\mathcal{X}$.

La proposition~1 est une sorte de compl\'{e}ment au th\'{e}or\`{e}me~3c
de~\cite{HeG-MiV2007} quand un plongement dans un espace affine complexe est
impos\'{e}~:\medskip

\begin{proposition}
[Sur un probl\`{e}me de Plateau complexe]\label{P/ harmo}\textit{On
consid\`{e}re dans }$\mathbb{C}^{n}$ \textit{une courbe r\'{e}elle }$\gamma
$\textit{ lisse, orient\'{e}e, }compacte et connexe. On fixe des champs de
vecteurs $\nu$ et $\tau$ d\'{e}finis sur $\gamma$ et de classe $C^{1}$ tels
que pour tout $x\in\gamma$, $\left(  \nu_{x},\tau_{x}\right)  $ est
orthonorm\'{e}e et $\tau_{x}$ d\'{e}finit l'orientation de $\gamma$ en $x$. On
note $z\mapsto z^{\ast}$ l'isomorphisme de $\mathbb{C}^{n}$ sur son dual
naturellement associ\'{e} \`{a} sa structure hermitienne standard. On se donne
alors des fonctions $u\in C^{1}\left(  \gamma,\mathbb{R}\right)  $ et
$u^{\prime}\in C^{0}\left(  \gamma,\mathbb{R}\right)  $ puis on pose
$\alpha=\frac{1}{2}\left(  u^{\prime}-iTu\right)  \left(  \nu^{\ast}%
+i\tau^{\ast}\right)  \in C^{0}\left(  \gamma,\Lambda^{1,0}T^{\ast}%
\mathbb{C}^{n}\right)  $. Si $\alpha\neq0$, a\textit{lors la condition}%
\begin{equation}
\forall h\in\mathcal{O}\left(  \mathbb{C}^{n}\right)  ,~\int_{\gamma}h\alpha=0
\label{Mom forme sur courbe}%
\end{equation}
\textit{est n\'{e}cessaire et suffisante \`{a} l'existence dans }%
$\mathbb{C}^{n}\backslash\gamma$ d'\textit{une courbe complexe connexe
}$\mathcal{X}$ \textit{de volume fini bord\'{e}e par }$\gamma$ \textit{\`{a}
laquelle }$u$\textit{\ se prolonge en une fonction harmonique multivalu\'{e}e
}$\widetilde{u}$ \textit{telle que }$\alpha$ est la valeur au bord de
$\partial\widetilde{u}$.
\end{proposition}

Il ne fait pas de doute que si $\gamma$ n'est plus suppos\'{e}e connexe et que
$\alpha\neq0$ sur chaque composante de $\gamma$, la conclusion devient que
$\mathcal{X}$ est une 1-chaine holomorphe et $\widetilde{u}$ une fonction
harmonique multivalu\'{e}e v\'{e}rifiant $\overline{\partial}\widetilde
{u}\wedge\mathcal{X}=\alpha\left[  \gamma\right]  $.

Dans l'\'{e}nonc\'{e} ci-dessus, $\left(  \gamma,u,\alpha\right)  $ est
r\'{e}alis\'{e} comme une donn\'{e}e de Cauchy du Laplacien d'une fonction
harmonique multivalu\'{e}e. Le fait que cette fonction soit univalu\'{e}e est
aussi d\'{e}termin\'{e} par une condition de moments qui s'exprime comme
dans~\cite[th. 3a/C]{HeG-MiV2007} \`{a} l'aide d'une fonction de Green.

Dans le cas d'une courbe complexe \'{e}ventuellement singuli\`{e}re
$\mathcal{Y}$ de $\mathbb{C}^{n}$ une fonction de Green est une fonction $g$
d\'{e}finie sur $\operatorname*{Reg}\mathcal{Y}\times\operatorname*{Reg}%
\mathcal{Y}$ priv\'{e} de sa diagonale telle que pour tout $z_{\ast}%
\in\operatorname*{Reg}\mathcal{Y}$, $g_{z_{\ast}}=g\left(  z_{\ast},.\right)
$ est harmonique sur $\mathcal{Y}^{\ast}$ (au sens pr\'{e}cis\'{e}
pr\'{e}c\'{e}demment) et qui au sens des courants v\'{e}rifie $i\partial
\overline{\partial}g_{z_{\ast}}=\delta_{z_{\ast}}dV$ o\`{u} $dV=i\partial
\overline{\partial}\left\vert .\right\vert ^{2}$ et $\delta_{z_{\ast}}$ est la
mesure de Dirac port\'{e}e par $\left\{  z_{\ast}\right\}  $.

\begin{proposition}
[Fonction de Green pour une courbe singuli\`{e}re]\label{P/ Green}Les
hypoth\`{e}ses et les notations \'{e}tant les m\^{e}mes que dans la
proposition~\ref{P/ harmo}, on consid\`{e}re le double $\mathcal{Z}$ de
$\mathcal{X}$. Alors il existe un domaine $\mathcal{Y}$ de $\mathcal{Z}$
contenant $\overline{\mathcal{X}}$ qui admet une fonction de Green $g$ et pour
que $\left(  \gamma,u,\alpha\right)  $ soit la donn\'{e}e de Cauchy du
Laplacien d'une fonction univalu\'{e}e, il suffit que
\[
0=\int_{\gamma}u\partial g\left(  z,.\right)  +g\left(  z,.\right)
\overline{\alpha}%
\]
pour tout $z\in\mathcal{Y}\backslash\overline{\mathcal{X}}$. Quand cette
condition est v\'{e}rifi\'{e}e, $\left(  \gamma,u,\alpha\right)  $ est la
donn\'{e}e de Cauchy du Laplacien de la fonction $U$ d\'{e}finie sur
$\operatorname*{Reg}\mathcal{X}$ par $U\left(  z\right)  =\frac{2}{i}%
\int_{\gamma}u\partial g\left(  z,.\right)  +g\left(  z,.\right)
\overline{\alpha}$.
\end{proposition}

L'existence de fonctions de Green pour une courbe singuli\`{e}re n'\'{e}tait
\`{a} notre connaissance pas connue. Ces fonctions qui jouent un r\^{o}le
crucial dans la preuve du th\'{e}or\`{e}me~\ref{T/ harmo} sont construites
\`{a} partir de formules explicites.\medskip

Le but principal de cet article est de prouver que contrairement \`{a} ce qui
se produit pour une courbe r\'{e}elle orpheline, les conditions de moments
impliqu\'{e}es sont automatiquement valid\'{e}es lorsque la courbe
consid\'{e}r\'{e}e appartient \`{a} une famille soumise \`{a} des conditions
naturelles si ce n'est g\'{e}n\'{e}riques. Le premier cas de validation
automatique de conditions de moments remonte \`{a} une d\'{e}monstration qu'a
faite Poincar\'{e} en 1907 dans~\cite{PoH1907} pour prouver le
th\'{e}or\`{e}me d'extension de Hartogs dans le cas particulier de la
sph\`{e}re. Ce ph\'{e}nom\`{e}ne a \'{e}t\'{e} ensuite observ\'{e} par Severi
dans un th\'{e}or\`{e}me concernant des familles de fonctions holomorphes, ou,
plus pr\'{e}cis\'{e}ment, de fonctions de classe $CR^{\omega}$ de
$\mathbb{E}^{n,m}=\mathbb{R}^{n}\times\mathbb{C}^{m}$ c'est \`{a} dire de
fonctions r\'{e}elles analytiques qui sont holomorphes par rapport \`{a} leur
$m$ variables complexes~:\medskip

\noindent\textbf{Th\'{e}or\`{e}me} (Severi-Brown-Bochner). \textit{Soient dans
}$\mathbb{E}^{n,m}$ ($n,m\in\mathbb{N}^{\ast}$) \textit{un domaine }$\Omega$
\textit{born\'{e} de bord connexe et }$f$\textit{,} \textit{une fonction\ de
classe }$CR^{\omega}$ \textit{au voisinage de} $b\Omega$. \textit{Alors} $f$
\textit{se prolonge en fonction du m\^{e}me type au voisinage de }%
$\overline{\Omega}$.\medskip

Ce r\'{e}sultat a \'{e}t\'{e} d\'{e}montr\'{e} en 1932 par
Severi~\cite{SeF1932} lorsque $m=1$ et que les sections complexes de $\Omega$
sont simplement connexes, restriction que Brown aurait lev\'{e}e
dans~\cite{BrA1936} en m\^{e}me temps qu'il aurait donn\'{e} la premi\`{e}re
d\'{e}monstration compl\`{e}te du th\'{e}or\`{e}me d'extension d\'{e}crit par
Hartogs en 1903. Si l'on se r\'{e}f\`{e}re \`{a} \cite{MeJ-PoE2007}, des
doutes demeurent quant aux arguments de Brown. Quoiqu'il en soit, Fueter puis
Bochner ont apport\'{e} une preuve de ces r\'{e}sultats dans~\cite{FuR1939} et
\cite[ch. 4, \S 2, th. 1 et th. 2]{BoS-MaW1948}\cite[th. 6]{BoS1954} (voir
aussi \cite{HeG-MiV2002}). Bochner~\cite{BoS1954} y a ajout\'{e} des versions
pour des fonctions CR-harmoniques, c'est \`{a} dire harmoniques sur chaque
section complexe de $\Omega$.

Dans le th\'{e}or\`{e}me de Severi-Brown, la fonction CR \`{a} prolonger est
r\'{e}elle analytique au voisinage du bord de son domaine d'extension.
L'\'{e}quivalent g\'{e}om\'{e}trique de cette hypoth\`{e}se est l'existence
d'un anneau r\'{e}el analytique l\'{e}vi-plat dans lequel est plong\'{e}e le
cycle qu'on cherche \`{a} r\'{e}aliser comme bord~; elle est analogue aux
hypoth\`{e}ses utilis\'{e}es par Rothstein~\cite{RoW1959}\cite{RoW-SpH1959}
pour r\'{e}aliser des cycles comme bord d'espace complexe. Un \textbf{anneau
r\'{e}el analytique l\'{e}vi-plat\ }de $\mathbb{E}^{n,2}$ est par
d\'{e}finition un sous-ensemble r\'{e}el analytique $N$ d'un ouvert de
$\mathbb{E}^{n,2}$ qui est localement la r\'{e}union de sous-vari\'{e}t\'{e}s
de $\mathbb{E}^{n,2}$ de dimension $n+2$, de classe $C^{\omega}$ et dont les
sections sont des courbes complexes de volume fini~; les composantes
irr\'{e}ductibles d'un anneau l\'{e}vi-plat en un point sont donc des
sous-vari\'{e}t\'{e}s de $\mathbb{E}^{n,2}$ de dimension $n+2$. Notons que
cette hypoth\`{e}se entra\^{\i}ne aussi que $N$ est coup\'{e}e
transversalement par les sous-espaces $\mathbb{C}_{t}^{2}=\left\{  t\right\}
\times\mathbb{C}^{2}$, $t\in\mathbb{R}^{n}$.

Notre premier th\'{e}or\`{e}me qui g\'{e}n\'{e}ralise g\'{e}om\'{e}triquement
celui de Severi-Brown utilise concerne les sous-vari\'{e}t\'{e}s $M$ de classe
$C^{k}$ ($k\in\mathbb{N}^{\ast}\cup\left\{  +\infty,\omega\right\}  $) de
$\mathbb{E}^{n,2}$ \textbf{plong\'{e}es} dans un anneau r\'{e}el analytique
l\'{e}vi-plat $N$, c'est \`{a} dire telles que pour tout $p\in M$, il existe
un voisinage $U$ de $p$ et une composante irr\'{e}ductible $\widetilde{N}_{p}$
de $N\cap U$ en $p$ tels que $M\cap U$ est une hypersurface de classe $C^{k}$
de $\widetilde{N}_{p}$.

On utilise les notations suivantes~: $E_{t}$ pour $E\cap\mathbb{C}_{t}^{2}$
lorsque $E\subset\mathbb{E}^{n,2}$ et $t\in\mathbb{R}^{n}$, $E^{t}$ pour la
projection naturelle de $E_{t}$ sur $\mathbb{C}^{2}$, $\mathcal{H}^{d}$ pour
d\'{e}signer la mesure $d$-dimensionnelle de Hausdorff et $\left[  V\right]  $
pour d\'{e}signer un courant d'int\'{e}gration sur $V$ quand $V$ est une
sous-vari\'{e}t\'{e} orient\'{e}e ou plus g\'{e}n\'{e}ralement un
sous-ensemble r\'{e}el analytique orient\'{e} d'un ouvert de $\mathbb{E}%
^{n,2}$.

\begin{theorem}
[Ph\'{e}nom\`{e}ne de Hartogs-Severi pour des feuilletages CR singuliers]%
\label{T/bord CxR A}On consid\`{e}re dans $\mathbb{E}^{n,2}$ une
sous-vari\'{e}t\'{e} $M$ de classe $C^{r}$, $r\in\mathbb{N}^{\ast}\cup\left\{
\infty,\omega\right\}  $, compacte, connexe, orient\'{e}e, de dimension $n+1$,
plong\'{e}e dans un anneau r\'{e}el analytique l\'{e}vi-plat $N$ et dont les
sections $M_{t}=M\cap\mathbb{C}_{t}^{2}$ sont des courbes r\'{e}elles de
longueur finie, \'{e}ventuellement r\'{e}duites \`{a} un point. Il existe
alors dans $\mathbb{E}^{n,2}\backslash M$ un unique sous-ensemble r\'{e}el
analytique orient\'{e} connexe $\mathcal{X}$ de dimension $n+2$, de volume
fini, dont les sections sont des courbes complexes de volume fini et tel que
$d\left[  \mathcal{X}\right]  =\pm\left[  M\right]  $~; lorsque $N$ est lisse,
$\mathcal{X}$ contient pr\`{e}s de $M$ une seule des deux composantes de
$N\backslash M$ dont $M$ est le bord. En dehors d'un compact $%
A%
$ de $M$ tel que $\mathcal{H}^{n+1}\left(
A%
\right)  =0$, $\overline{\mathcal{X}}\backslash%
A%
$ est localement une vari\'{e}t\'{e} \`{a} bord de classe $C^{r}$ de bord $M$.
L'ensemble singulier de $\mathcal{X}$ est de $\mathcal{H}^{n}$-mesure finie et
pour tout $t\in\mathbb{R}^{n}$ tel que $\mathcal{H}^{1}\left(  M_{t}\right)
>0$, $\mathcal{X}_{t}$ est une courbe complexe de $\mathbb{C}_{t}%
^{2}\backslash M_{t}$ dont l'ensemble singulier est $\operatorname{Sing}%
\mathcal{X}_{t}=\left(  \operatorname{Sing}\mathcal{X}\right)  _{t}$. Enfin,
$\mathcal{X}$ est un ensemble r\'{e}el analytique coh\'{e}rent au sens de
Cartan~\cite{CaH1957}.
\end{theorem}

\noindent\textbf{Remarques. }

\textbf{1. }$\overline{\mathcal{X}}$ n'est pas n\'{e}cessairement une
vari\'{e}t\'{e} \`{a} bord car le ph\'{e}nom\`{e}ne d\'{e}crit
dans~\cite[\S \ 10]{HaR-LaB1975} et \cite[p. 346]{HaR1977} est possible~: il
se peut que dans un voisinage $V$ d'un point $p$ de $M$, $\overline
{\mathcal{X}}$ contienne non seulement une sous-vari\'{e}t\'{e} \`{a} bord de
bord $M\cap V$ mais aussi une sous-vari\'{e}t\'{e} coupant $M\cap V$
uniquement en $p$.\smallskip

\textbf{2. }Dans \cite{CaH1957}, Cartan donne des exemples d'ensembles
r\'{e}els analytiques qui ne peuvent pas \^{e}tre d\'{e}finis par des
\'{e}quations r\'{e}elles analytiques globales. Il caract\'{e}rise cette
propri\'{e}t\'{e} par la coh\'{e}rence du faisceau d'id\'{e}aux naturellement
attach\'{e} \`{a} l'ensemble consid\'{e}r\'{e} ou encore \`{a} l'existence
d'un complexifi\'{e} global. La coh\'{e}rence de $\mathcal{X}$ qui par
ailleurs implique celle de $N$ permet donc d'affirmer que $\mathcal{X}$ peut
\^{e}tre d\'{e}fini par une \'{e}quation r\'{e}elle analytique globale et
$\mathcal{X}$ poss\`{e}de dans $\mathbb{E}^{n,2}$ un voisinage qui dans chaque
$\mathbb{C}_{t}^{2}$ est un voisinage de Stein de $\mathcal{X}_{t}$.\medskip

Le th\'{e}or\`{e}me~\ref{T/bord CxR A} entra\^{\i}ne que pour des fonctions
qui se pr\'{e}sentent en famille non seulement la condition des moments
n\'{e}cessaire dans la proposition~\ref{P/ harmo} est automatiquement
valid\'{e}e mais qu'en outre, les prolongements sont univalu\'{e}s. Le
th\'{e}or\`{e}me ci-dessous g\'{e}n\'{e}ralise et pr\'{e}cise celui de Bochner
sur l'extension des fonctions CR-harmoniques dans \cite[th. 6]{BoS1954}.

\begin{theorem}
[Ph\'{e}nom\`{e}ne de Bochner pour des feuilletages CR singuliers]%
\label{T/ harmo}Les hypoth\`{e}ses et les notations \'{e}tant celles du
th\'{e}or\`{e}me~\ref{T/bord CxR A}, on suppose en outre que $N$ est lisse. On
se donne alors une fonction r\'{e}elle $u$ qui est r\'{e}elle analytique et
CR-harmonique au voisinage de $M$ dans $N$. Alors\thinspace$u$ se prolonge
\`{a} $\mathcal{X}$ en une fonction CR-harmonique et analytique r\'{e}elle au
voisinage de tout point de $\operatorname*{Reg}\mathcal{X}_{t}$ quand cet
ensemble n'est pas vide.
\end{theorem}

Dans cet \'{e}nonc\'{e}, une fonction $u$ est dite r\'{e}elle analytique au
voisinage de $M$ dans $N$ si pour chaque point $p$ de $M$, $u$ est r\'{e}elle
analytique dans un voisinage de $p$ dans la sous-vari\'{e}t\'{e} qui est la
composante $N_{p}$ de $N$ en $p$ dans laquelle $M$ est plong\'{e}e au
voisinage de $p$~; $u$ est dite CR-harmonique si en outre pour chaque
$p=\left(  t,z\right)  \in M$, $u\left(  t,.\right)  $ est harmonique au
voisinage de $z$ dans la projection naturelle sur $\mathbb{C}^{2}$ de $N_{p}$.
Une fonction $U$ sur $\mathcal{X}$ est dite CR-harmonique si pour tout
param\`{e}tre $t$ tel que $\mathcal{H}^{1}\left(  M_{t}\right)  >0$, $U\left(
t,.\right)  $ est harmonique sur $\mathcal{X}_{t}$~; l'analyticit\'{e} de $U$
par rapport au param\`{e}tre s'entend au voisinage de tout point $\left(
t,z\right)  $ de $\operatorname*{Reg}\mathcal{X}_{t}$.\medskip

Le th\'{e}or\`{e}me~\ref{T/bord CxR A} s'inscrit dans la probl\'{e}matique de
r\'{e}aliser un cycle CR comme le bord d'une cha\^{\i}ne CR. Pour \^{e}tre
plus pr\'{e}cis, consid\'{e}rons un ouvert $U$ de l'espace $\mathbb{E}^{n,m}$
orient\'{e} par $dt\wedge\underset{1\leqslant j\leqslant m}{\wedge}%
i\,dz_{j}\wedge d\overline{z_{j}}$ o\`{u} $dt=dt_{1}\wedge\cdots\wedge dt_{n}%
$. Si $\mathcal{X}$ est une partie de $U$ pour laquelle il existe un ferm\'{e}
$%
S%
$ contenu dans $\mathcal{X}$ tel que $\mathcal{H}^{d}\left(
S%
\right)  =0$ et $\mathcal{X}\backslash%
S%
$ est une sous-vari\'{e}t\'{e} $CR$ de dimension $d$ et de classe $C^{r}$,
$r\in\mathbb{N}^{\ast}\cup\left\{  +\infty,\omega\right\}  $, on dit que
$\mathcal{X}$ est un \textbf{sous-ensemble} $CR$ de classe $C^{r}$ de $U$
\`{a} singularit\'{e}s n\'{e}gligeables et on pose $\dim_{CR}\mathcal{X}%
=\dim_{CR}\mathcal{X}\backslash%
S%
$~; $\operatorname*{Sing}\mathcal{X}$ est le plus petit ferm\'{e} $%
S%
$ ayant cette propri\'{e}t\'{e}. Si en outre $\mathcal{X}\backslash%
S%
$ est une vari\'{e}t\'{e} orient\'{e}e de volume fini, $\left[  \mathcal{X}%
\backslash%
S%
\right]  $ est not\'{e} abusivement mais plus simplement, $\left[
\mathcal{X}\right]  $ et $\mathcal{X}$ est dit presque orient\'{e}. Quand
$\mathcal{H}^{d-1}\left(
S%
\right)  =0$, cette notation n'est plus ambigu\"{e} et on dit que
$\mathcal{X}$ est orient\'{e}~; $\mathcal{X}$ est orientable lorsque par
exemple les conditions suivantes sont r\'{e}unies~: $d=n+2$, $\mathcal{X}$ est
feuillet\'{e} par des courbes complexes, son volume est fini et sa partie
singuli\`{e}re $\mathcal{X}^{s}$ v\'{e}rifie $\mathcal{H}^{n+1}\left(
\mathcal{X}^{s}\right)  =0$~: si $p=\left(  t,z\right)  $ est un point
r\'{e}gulier de $\mathcal{X}$, $\mathcal{X}$ est au voisinage de $p$
transversalement sectionn\'{e} par $\mathbb{C}_{t}^{2}$ en une courbe complexe
de sorte que $idt\wedge\left(  dz_{1}\wedge d\overline{z_{1}}+dz_{2}\wedge
d\overline{z_{2}}\right)  $ est une forme volume pour $\mathcal{X}%
\backslash\mathcal{X}^{s}$~; on convient dans ce cas que cette forme volume
d\'{e}finit l'orientation naturelle de $\mathcal{X}$ et que $\left[
\mathcal{X}\right]  $ est son courant d'int\'{e}gration associ\'{e}.

On d\'{e}finit une \textbf{(d,k)-cha\^{\i}ne-CR} de $U$ comme \'{e}tant une
combinaison lin\'{e}aire \`{a} coefficients dans $\mathbb{Z}$ de courants
d'int\'{e}gration de sous-ensembles $CR$ de $U$ presque orient\'{e}s de
dimension $d$ et de dimension $CR$ $k$~; la cha\^{\i}ne est dite orient\'{e}e
si la combinaison lin\'{e}aire pr\'{e}c\'{e}dente ne fait intervenir que des
courants d'int\'{e}gration sur des sous-ensembles $CR$ orient\'{e}s~; un
\textbf{(d,k)-cycle-CR} est une $\left(  d,k\right)  $-cha\^{\i}ne-$CR$
ferm\'{e}e au sens des courants. Un probl\`{e}me naturel dans cette situation
est de savoir si un $\left(  d,k\right)  $-cycles-CR $T$ donn\'{e} de masse
finie peut \^{e}tre r\'{e}alis\'{e} comme le bord dans $\mathbb{E}^{n,m}$ de
l'extension simple d'une $\left(  d+1,k+1\right)  $-cha\^{\i}ne-$CR$ de masse
finie de $\mathbb{E}^{n,m}\backslash M$ o\`{u} $M$ est le support de
$T$.\medskip

Lorsque dans $\mathbb{E}^{1,m}$, $m\geqslant3$, on consid\`{e}re un $\left(
2m-1,m-1\right)  $-cycle $T$ de masse finie dont le support $M$ est une
sous-vari\'{e}t\'{e} $C^{\infty}$ \`{a} singularit\'{e}s n\'{e}gligeables,
Dolbeault, Tomassini et Zaitsev ont prouv\'{e} dans~\cite{DoP-ToG-ZaD2009} que
$T$ est d'une unique fa\c{c}on le bord dans $\mathbb{E}^{1,m}$ de l'extension
simple d'une $\left(  2m,m\right)  $-cha\^{\i}ne $CR$ de masse finie de
$\mathbb{E}^{1,m}\backslash M$~; ils utilisent ce r\'{e}sultat pour obtenir
une condition qui donne l'existence et l'unicit\'{e} d'une hypersurface
l\'{e}vi-plate de $\mathbb{C}^{n}$, $n>2$, dont le bord est un compact
prescrit. Le cas d'une cha\^{\i}ne est trait\'{e} dans~\cite{DoP1986} quand
celle-ci est la trace sur $\mathbb{E}^{1,m}$ du support d'un cycle
maximalement complexe de $\mathbb{C}\times\mathbb{C}^{m}$.

Lorsque $M$ est essentiellement une famille de courbes r\'{e}elles, c'est
\`{a} dire quand $T$ est un $\left(  2,0\right)  $-cycle-CR de masse finie de
$\mathbb{E}^{1,2}$, il est \'{e}nonc\'{e} dans~\cite{DoP1994} que $T$ est le
bord\ d'une $\left(  3,1\right)  $-cha\^{\i}ne-CR s'il v\'{e}rifie une
condition de moments CR. Dans les deux r\'{e}f\'{e}rences cit\'{e}es la
$\left(  d+1,k+1\right)  $-cha\^{\i}ne-CR trouv\'{e}e pour border le $\left(
d,k\right)  $-cycle-$CR$ donn\'{e} est feuillet\'{e}e par des $\left(
k+1\right)  $-cha\^{\i}nes holomorphes.\medskip

Dans le th\'{e}or\`{e}me ci-apr\`{e}s, on consid\`{e}re un $\left(
n+1,0\right)  $-cycle dont le support est un sous-ensemble r\'{e}el analytique
$M$ de $\mathbb{E}^{n,2}$ \textbf{plong\'{e}} dans un anneau r\'{e}el
analytique l\'{e}vi-plat $N$, ce qui signifie que pour tout $p\in M$ et toute
composante irr\'{e}ductible $\widetilde{M}$ de $M$ en $p$, il existe un
voisinage $U$ de $p$ et une composante irr\'{e}ductible $\widetilde{N}$ de $N$
en $p$ tels que $\widetilde{M}\cap U$ est un sous-ensemble r\'{e}el analytique
de la sous-vari\'{e}t\'{e} $\widetilde{N}$. Si $M$ est une
sous-vari\'{e}t\'{e} de $\mathbb{E}^{n,2}$ et si $M$ est coup\'{e}e
transversalement par $\mathbb{C}_{t}^{2}=\left\{  t\right\}  \times
\mathbb{C}^{2}$ ($t\in\mathbb{R}^{n}$) le long de $M_{t}=M\cap\mathbb{C}%
_{t}^{2}$, on dit que $M_{t}$ est une section r\'{e}guli\`{e}re de $M$~;
l'ensemble des param\`{e}tres $t$ pour lesquels $M_{t}$ est une section
r\'{e}guli\`{e}re de $V$ est not\'{e} $\mathcal{T}\left(  M\right)  $. Lorsque
$M$ est un sous-ensemble r\'{e}el analytique de $\mathbb{E}^{n,2}$ plong\'{e}
dans un anneau r\'{e}el analytique l\'{e}vi-plat $N$, $\mathcal{T}\left(
M\right)  $ d\'{e}signe l'ensemble des param\`{e}tres $t$ tels que chaque
composante locale de $M$ est coup\'{e}e transversalement par $\mathbb{C}%
_{t}^{2}$.

Les sections d'un courant sont comprises au sens de~\cite[sec. 4.3]%
{FeH1969Li}~; si $\mu$ est un courant localement plat de $\mathbb{E}^{n,2}$ de
dimension $d\geqslant n$, on peut d\'{e}finir un courant $\mu_{t_{\ast}}$ de
dimension $d-n$ pour presque tout $t_{\ast}$ dans $\mathbb{R}^{n}$ par la
formule%
\begin{equation}
\forall\varphi\in%
D%
_{d-n}\left(  \mathbb{E}^{n,2}\right)  ,~\left\langle \mu_{t_{\ast}}%
,\varphi\right\rangle =\ \underset{r\downarrow0^{+}}{\lim}\frac{1}{c_{n}r^{n}%
}\left\langle \mu,\mathbf{1}_{B_{\mathbb{R}^{n}}\left(  t_{\ast},r\right)
}dt\wedge\varphi\right\rangle \label{F/ courant integ}%
\end{equation}
o\`{u} $%
D%
_{d-n}\left(  \mathbb{E}^{n,2}\right)  $ est l'espace des $\left(  d-n\right)
$-formes de classe $C^{\infty}\left(  \mathbb{E}^{n,2}\right)  $ \`{a} support
compact, $c_{n}$ est le volume de la boule unit\'{e} de $\mathbb{R}^{n}$~;
$\operatorname*{Spt}\mu_{t_{\ast}}\subset\mathbb{C}_{t_{\ast}}^{2}%
\cap\operatorname*{Spt}\mu$ et $d\mu_{t_{\ast}}=\left(  -1\right)  ^{n}\left(
d\mu\right)  _{t_{\ast}}$ si $\left(  d\mu\right)  _{t_{\ast}}$ est aussi d\'{e}fini.

\begin{theorem}
[Ph\'{e}nom\`{e}ne de Hartogs-Severi pour des cycles CR]\label{T/bord CxR B}%
Soit $T$ un $\left(  n+1,0\right)  $-cycle $CR$ de $\mathbb{E}^{n,2}$ dont le
support $M$ est un sous-ensemble r\'{e}el analytique compact de $\mathbb{E}%
^{n,2}$ plong\'{e} dans un anneau r\'{e}el analytique l\'{e}vi-plat $N$ de
$\mathbb{E}^{n,2}$. Il existe alors dans $\mathbb{E}^{n,2}\backslash M$ une
unique $\left(  n+2,1\right)  $-cha\^{\i}ne $CR$ orient\'{e}e $X$ de masse
finie telle que $dX=T$.

De plus, pour tout $t\in\mathbb{R}^{n}$, $X_{t}$ est une $1$-cha\^{\i}ne
holomorphe de $\mathbb{C}_{t}^{2}\backslash M_{t}$ de masse finie telle que
$dX_{t}=T_{t}$ dont le lieu singulier est la trace sur $\mathbb{C}_{t}^{2}$ de
$\operatorname{Sing}\mathcal{X}$.

En outre, pour tout point $p$ de $M$, il existe un voisinage $W$ de $p$ tel
que $W\cap\mathcal{X}$ contienne au moins l'une des composantes connexes de
$(W\cap N)\backslash M$~; si $p=\left(  t,z\right)  $ est un point
r\'{e}gulier de $M$ o\`{u} $M$ est transverse \`{a} $\mathbb{C}_{t}^{2}$ et si
$W$ est suffisamment petit, chaque composante connexe de $\overline
{\mathcal{X}}\cap W$ est soit une vari\'{e}t\'{e} \`{a} bord r\'{e}elle
analytique de bord $M\cap W$, soit un sous-ensemble r\'{e}el analytique de $W$.

Enfin, si $M\backslash\operatorname*{Sing}M$ est connexe, alors $\mathcal{X}$
est irr\'{e}ductible et $X\in\mathbb{Z}\left[  \mathcal{X}\right]  $.
\end{theorem}

\section{Preuves des r\'{e}sultats}

\subsection{Preuve du th\'{e}or\`{e}me~\ref{T/bord CxR A}}

Soit $N$ un anneau r\'{e}el analytique l\'{e}vi-plat de $\mathbb{E}^{n,2}$ et
$M$ une sous-vari\'{e}t\'{e} de $\mathbb{E}^{n,2}$ plong\'{e}e dans $N$ et
satisfaisant aux hypoth\`{e}ses du th\'{e}or\`{e}me~\ref{T/bord CxR A}. On
peut donc s\'{e}lectionner un voisinage ouvert $N^{o}$ de $M$ dans $N$ tel que
$N^{o}\backslash M$ a exactement deux composantes connexes $N^{+}$ et $N^{-}%
$qui sont des vari\'{e}t\'{e}s ouvertes \`{a} bord de bord respectifs
$M^{+}\cup M$ et $M^{-}\cup M$, $M^{+}$ et $M^{-}$ \'{e}tant des
vari\'{e}t\'{e}s de classe $C^{r}$ plong\'{e}es dans $N$ et bien s\^{u}r
connexes, compactes, orient\'{e}es. Afin de fixer les id\'{e}es, on convient
de noter $N^{+}$ la composante dont le bord est $\left[  M\right]  -\left[
M^{+}\right]  $.

On note $\pi_{\mathbb{R}}$ la projection naturelle de $\mathbb{E}^{n,2}$ sur
son premier facteur et $\pi_{\mathbb{C}}$ celle sur son second facteur. Le
lemme \ref{L/ chgt var} ci dessous permet de traiter une section donn\'{e}e de
$M$ comme une section r\'{e}guli\`{e}re.

\begin{lemma}
\label{L/ chgt var}Pour tout $t\in\pi_{\mathbb{R}}\left(  M\right)  $, il
existe dans tout voisinage $U$ de $M$ dans $\mathbb{E}^{n,2}$ une
sous-vari\'{e}t\'{e} ouverte \`{a} bord $N^{t}$ de bord $M^{t}\cup M$ telle
que $N^{t}\subset\overline{N^{+}}\cap U$, $M^{t}$ est une sous-vari\'{e}t\'{e}
connexe compacte orient\'{e}e de classe $C^{\omega}$ plong\'{e}e dans $N$,
$d\left[  N^{t}\right]  =\left[  M\right]  -\left[  M^{t}\right]  $ et
$t\in\mathcal{T}\left(  M^{t}\right)  $.
\end{lemma}

\begin{proof}
Fixons un param\`{e}tre $t$ dans $\pi_{\mathbb{R}}\left(  M\right)  $ et
supposons d'abord que $M$ est une sous-vari\'{e}t\'{e} de dimension $n+1$
plong\'{e}e dans $N$ et que $N$ est une sous-vari\'{e}t\'{e} d'un ouvert de
$\mathbb{E}^{n,2}$. Fixons alors un voisinage $G$ de $M$ dans $\mathbb{E}%
^{n,2}$ tel que la fonction $\rho$ de $N\cap G$ dans $\mathbb{R}$ qui \`{a}
$p\in N$ associe $\pm dist_{\mathbb{E}^{n,2}}\left(  p,M\right)  $ si
$p\in\overline{N^{\pm}}\cap G$ est de classe $C^{1}$ et diff\'{e}rentielle ne
s'annulant pas. Le th\'{e}or\`{e}me de Sard appliqu\'{e} \`{a} $\rho\left(
.,t\right)  $ fournit $\lambda\in\mathbb{R}_{+}^{\ast}$ arbitrairement petit
tel que $\left\{  \rho=\lambda\right\}  $ convient a les propri\'{e}t\'{e}s
requises hormis peut \^{e}tre l'analyticit\'{e} r\'{e}elle. Ce manque
\'{e}ventuel est combl\'{e} en consid\'{e}rant une approximation $C^{1}$ de
$\left\{  \rho=\lambda\right\}  $ par une vari\'{e}t\'{e} r\'{e}elle
analytique $M^{\prime}$ suffisamment proche.

Supposons maintenant que $N$ \'{e}tant seulement un anneau r\'{e}el analytique
l\'{e}vi-plat, $M$ est une hypersurface plong\'{e}e dans $N$. On consid\`{e}re
un rev\^{e}tement de $N$ par une vari\'{e}t\'{e} lisse $\widetilde{N}$~; on
note $\varphi$ la projection naturelle de $\widetilde{N}$ sur $N$.
$\varphi^{-1}\left(  M\right)  $ est alors une hypersurface lisse de
$\widetilde{N}$. Par ailleurs, les sections $N_{t}$ de $N$ \'{e}tant par
hypoth\`{e}se des courbes complexes, $N$ et $\mathbb{C}_{t}^{2}$ se coupent
transversalement et $\varphi^{-1}\left(  N_{t}\right)  $ est donc lisse. Dans
des voisinages arbitrairement petits de $\varphi^{-1}\left(  M\right)  $ dans
$\widetilde{N}$, une approximation r\'{e}elle analytique g\'{e}n\'{e}rique
$\widehat{M}$ de $\varphi^{-1}\left(  M\right)  $ coupe donc $\varphi
^{-1}\left(  N_{t}\right)  $ transversalement et a la propri\'{e}t\'{e} que
$\left[  \varphi^{-1}\left(  M\right)  \right]  -[\,\widehat{M}\,]$ est le
bord au sens de la formule de Stokes d'un domaine lisse et born\'{e} de
$\widetilde{N}$, ce qui entra\^{\i}ne que $\left[  M\right]  -[\varphi
(\,\widehat{M}\,)]$ est le bord au sens de la formule de Stokes d'un ouvert
born\'{e} de $N$. $\varphi$ \'{e}tant un diff\'{e}omorphisme local, on en
d\'{e}duit que $M^{t}=\varphi(\,\widehat{M}\,)$ coupe transversalement
$\mathbb{C}_{t}^{2}$.\smallskip
\end{proof}

\begin{corollary}
\label{L/ ExistSect}Pour tout $t\in\pi_{\mathbb{R}}\left(  M\right)  $, les
intersections de $\mathbb{C}_{t}^{2}$ avec $M^{t}$ et $N^{t}$ sont
transverses, ce qui permet de d\'{e}finir $\left[  M\right]  _{t}$ par la
formule $\left[  M\right]  _{t}=\left[  M^{t}\right]  _{t}+d\left[
N^{t}\right]  _{t}$ o\`{u} $\left[  M^{t}\right]  _{t}$ et $\left[
N^{t}\right]  _{t}$ sont d\'{e}finis de mani\`{e}re usuelle.
\end{corollary}

\begin{proof}
Lorsque $t\in\mathcal{T}\left(  M\right)  $, $M$ et $\mathbb{C}_{t}^{2}$ se
coupent transversalement de sorte que $\left[  M\right]  _{t}$ est bien
d\'{e}fini et vaut $\left[  M_{t}\right]  $. Lorsque $t\in\pi_{\mathbb{R}%
}\left(  M\right)  $ est quelconque, $t$ est par construction un param\`{e}tre
r\'{e}gulier de $M^{t}$ de sorte que $\left[  M^{t}\right]  _{t}$ est bien
d\'{e}fini et vaut $\left[  \left(  M^{t}\right)  _{t}\right]  $. Par
ailleurs, les sections $N_{\tau}$ \'{e}tant par hypoth\`{e}se des courbes
complexes, $N$ et $\mathbb{C}_{t}^{2}$ se coupent transversalement, ce qui
implique que $\left[  N\right]  _{t}$ a un sens et co\"{\i}ncide avec $\left[
N_{t}\right]  $. Lorsque $t\in\mathcal{T}\left(  M\right)  $, la formule
$\left[  M\right]  _{t}=\left[  M^{t}\right]  _{t}+d\left[  N^{t}\right]
_{t}$ d\'{e}coule donc directement de l'\'{e}galit\'{e} $\left[  M\right]
-\left[  M^{t}\right]  =d\left[  N^{t}\right]  $. Lorsque $t\in\pi
_{\mathbb{R}}\left(  M\right)  \backslash\mathcal{T}\left(  M\right)  $, elle
permet de d\'{e}finir $\left[  M\right]  _{t}$ de fa\c{c}on coh\'{e}rente.
\end{proof}

Consid\'{e}rons comme dans \cite{HeG-MiV2002} la forme de Poincar\'{e}~; pour
$t$ fix\'{e} dans $\mathbb{R}^{n}$, c'est la forme diff\'{e}rentielle
r\'{e}elle $P_{t}$ de degr\'{e} $n-1$ d\'{e}finie sur $\mathbb{R}%
^{n}\backslash\left\{  t\right\}  $ par la formule%
\[
P_{t}=\frac{1}{c_{n}}\sum\limits_{1\leqslant j\leqslant n}\left(  -1\right)
^{j-1}\frac{\theta_{j}-t_{j}}{\left\vert \theta-t\right\vert ^{n}}%
d\theta^{\widehat{j}}\
\]
D'o\`{u} $d\theta^{\widehat{j}}=\underset{\nu\neq j}{\wedge}d\theta_{\nu}~$;
$\int_{\left\vert \theta-t\right\vert =1}P_{t}\left(  \theta\right)  =1$,
$P_{t}$ est ferm\'{e}e sur $\mathbb{R}^{n}\backslash\left\{  t\right\}  $ et
$dP_{t}=\delta_{t}d\theta_{1}\wedge...\wedge d\theta_{n}$ o\`{u} $\delta_{t}$
est la mesure de Dirac en $t$. La singularit\'{e} de $P_{t}$ \'{e}tant
int\'{e}grable sur $M$, on peut poser, $h$ \'{e}tant une $\left(  1,0\right)
$-forme diff\'{e}rentielle fix\'{e}e de classe $CR^{\omega}$ sur
$\mathbb{E}^{n,2}$,%
\begin{equation}
I_{M,h}\left(  t\right)  =\int\limits_{\left(  \theta,z\right)  \in M}\partial
h\left(  \theta,z\right)  \wedge P_{t}\left(  \theta\right)
\label{T/bord CxR mom2}%
\end{equation}

\begin{lemma}
\label{L/reecrire mom}Pour tout $t\in\mathcal{T}\left(  M\right)  $, on a%
\begin{equation}
I_{M,h}\left(  t\right)  =\int_{M_{t}}h~. \label{T/bord CxR mom2reg}%
\end{equation}

\end{lemma}

Soit $t\in\mathcal{T}\left(  M\right)  $. Puisque $\mathbb{C}_{t}^{2}$ coupe
transversalement $M$ le long de $M_{t}$, pour $\varepsilon\in\mathbb{R}%
_{+}^{\ast}$ suffisamment petit $M\cap\left[  B_{\mathbb{R}^{n}}\left(
t,\varepsilon\right)  \times\mathbb{C}^{2}\right]  $ est une vari\'{e}t\'{e}
\`{a} bord de bord $\sigma_{t}^{\varepsilon}=\underset{\theta\in S\left(
t,\varepsilon\right)  }{\cup}M_{\theta}$ et $M_{t}^{\varepsilon}%
=M\backslash\left[  B_{\mathbb{R}^{n}}\left(  t,\varepsilon\right)
\times\mathbb{C}^{2}\right]  $ est un ensemble analytique orient\'{e}
bord\'{e} (topologiquement) par $\sigma_{t}^{\varepsilon}$. Comme $\partial
h\wedge P_{t}=d\left(  h\wedge P_{t}\right)  $ au voisinage de $\overline
{M_{t}^{\varepsilon}}$, la formule de Stockes livre%
\begin{equation}
\int\nolimits_{M_{t}^{\varepsilon}}\partial h\wedge P_{t}=\int
\nolimits_{\sigma_{t}^{\varepsilon}}h\wedge P_{t}. \label{T/bord CxR mom2'}%
\end{equation}
La singularit\'{e} de $P_{t}$ \'{e}tant int\'{e}grable sur $M$, le membre de
gauche de (\ref{T/bord CxR mom2'}) tend vers l'int\'{e}grale de $\partial
h\wedge P_{t}$ sur $M$ lorsque $\varepsilon$ tend vers $0$. Par ailleurs, la
forme particuli\`{e}re de $h\wedge P_{t}$ permet d'utiliser le
th\'{e}or\`{e}me de Fubini et d'obtenir%
\[
\int\nolimits_{\sigma_{t}^{\varepsilon}}h\wedge P_{t}=\int\nolimits_{\theta\in
S\left(  0,1\right)  }I_{M,h}\left(  t+\varepsilon\theta\right)  P_{t}\left(
\theta\right)  .
\]
En passant \`{a} limite lorsque $\varepsilon$ tend vers $0$, il appara\^{\i}t
que $I_{M,h}\left(  t\right)  $ tend vers le membre de droite de
(\ref{T/bord CxR mom2reg}). Le lemme est prouv\'{e}.

\begin{lemma}
\label{L/ mom = 0}$I_{M,h}$ est la fonction nulle et pour tout $t\in
\mathbb{R}^{n}$, $0=I_{M,h}\left(  t\right)  =\left\langle \left[  M\right]
_{t},h\right\rangle $.
\end{lemma}

\begin{proof}
Soient $t$ et $s$ deux param\`{e}tres fix\'{e}s dans $\pi_{\mathbb{R}}\left(
M\right)  $. On se donne $M^{t}$ et $G^{t}$ comme dans le
lemme~\ref{L/ chgt var}. Avec la formule de Stokes on obtient%
\[
\int_{M}\partial h\wedge P_{s}=\int_{M^{t}}\partial h\wedge P_{s}+\int_{G^{t}%
}\partial h\wedge\delta_{s}dx_{1}\wedge...\wedge dx_{n}=\int_{M^{t}}\partial
h\wedge P_{s}+\left\langle \left[  G^{t}\right]  _{s},\partial h\right\rangle
.
\]
Or le support de $\left[  G^{t}\right]  _{s}$ est contenu dans $N_{s}$ qui est
une courbe complexe et $\partial h$ est une $\left(  2,0\right)  $-forme
diff\'{e}rentielle holomorphe. Par cons\'{e}quent $\left\langle \left[
G^{t}\right]  _{s},\partial h\right\rangle =0$ et il s'av\`{e}re que
$I_{M,h}\left(  s\right)  =I_{M^{t},h}\left(  s\right)  $. Puisque
$t\in\mathcal{T}\left(  M^{t}\right)  $ la formule (\ref{T/bord CxR mom2reg})
s'applique \`{a} $M^{t}$ au lieu de $M$ pour $s$ voisin de $t$, ce qui donne
$I_{M,h}\left(  s\right)  =I_{M^{t},h}\left(  s\right)  =\int_{\left(
M^{t}\right)  _{s}}h$. $M^{t}$ \'{e}tant r\'{e}el analytique et coup\'{e}
transversalement par $\mathbb{C}_{s}^{2}$ lorsque $s$ est suffisamment voisin
de $t$, on en d\'{e}duit que $I_{M,h}$ est r\'{e}elle analytique au voisinage
de $t$. Il s'ensuit que $I_{M,h}\in C^{\omega}\left(  \mathbb{R}^{n}\right)  $
puis que $I_{M,h}$ est la fonction nulle car si $t\in\mathbb{R}^{n}%
\backslash\pi_{\mathbb{R}}\left(  M\right)  $, $M_{t}=\emptyset$ et la
relation $\partial h\wedge P_{t}=d\left(  hP_{t}\right)  $ est vraie dans un
voisinage de $M$. La formule annonc\'{e}e dans le lemme d\'{e}coule
directement des relations $\left[  M^{t}\right]  _{t}=\left[  M\right]
_{t}+d\left[  G^{t}\right]  _{t}$ et $I_{M,h}\left(  t\right)  =I_{M^{t}%
,h}\left(  t\right)  $.\medskip
\end{proof}

Lorsque $t\in\mathcal{T}\left(  M\right)  $, $M$ est coup\'{e}e
transversalement par $\mathbb{C}_{t}^{2}$ et assimilant avec un l\'{e}ger abus
de langage $\left[  M\right]  _{t}=\left[  M_{t}\right]  $ \`{a} un courant de
$\mathbb{C}_{t}^{2}$, on obtient un MC-cycle au sens de \cite{HaR-LaB1975}.
Puisque $0=I_{M,h}\left(  t\right)  =\int_{M_{t}}h$, ce cycle v\'{e}rifie la
condition des moments. $\mathbb{C}_{t}^{2}\backslash\operatorname*{supp}%
\left[  M\right]  _{t}$ contient par cons\'{e}quent
d'apr\`{e}s~\cite{HaR-LaB1975} une $1$-cha\^{\i}ne holomorphe $X_{t}$ de masse
finie v\'{e}rifiant $dX_{t}=\left[  M\right]  _{t}$. On sait en outre
d'apr\`{e}s~\cite[p. 411]{LaM1995}\cite[prop. 4.4]{DiT1998a} que si
$B_{\mathbb{E}^{n,2}}\left(  0,A\right)  $ contient $M$, la masse de $X_{t}$
est au plus $2A\mathcal{H}^{1}\left(  M_{t}\right)  $~; elle est donc
localement major\'{e}e par rapport \`{a} $t$ dans $\mathcal{T}\left(
M\right)  $.

Si $t\in\pi_{\mathbb{R}}\left(  M\right)  \backslash\mathcal{T}\left(
M\right)  $, les m\^{e}mes arguments donnent l'existence dans $\mathbb{C}%
_{t}^{2}\backslash\operatorname*{supp}\left[  M^{t}\right]  _{t}$ d'une
$1$-cha\^{\i}ne holomorphe $X_{t}^{t}$ de masse finie v\'{e}rifiant
$dX_{t}^{t}=\left[  M^{t}\right]  _{t}$~; compte tenu du
lemme~\ref{L/ ExistSect}, en posant $X_{t}=X_{t}^{t}+\left[  N^{t}\right]
_{t}$, on obtient dans $\mathbb{C}_{t}^{2}\backslash\operatorname*{supp}%
\left[  M\right]  _{t}$ une $1$-cha\^{\i}ne holomorphe de masse finie $X_{t}$
telle que $dX_{t}=\left[  M\right]  _{t}$.

Si $t\in\mathbb{R}^{n}$ on note $\mathcal{X}_{t}$ le support de la
1-cha\^{\i}ne $X_{t}$ et on pose
\begin{equation}
\widetilde{\mathcal{X}}=\underset{t\in\mathcal{T}\left(  M\right)  }{\cup
}\mathcal{X}_{t}~,\text{~}~\mathcal{X}=\underset{t\in\mathbb{R}^{n}}{\cup
}\mathcal{X}_{t}\text{~}. \label{F/ def X}%
\end{equation}
V\'{e}rifions tout d'abord l'absence de pathologie~:

\begin{lemma}
\label{L/ englob demianneau}$N^{+}\subset\mathcal{X}$.
\end{lemma}

\begin{proof}
Soit $t\in\pi_{\mathbb{R}}\left(  M\right)  $. Supposons tout d'abord que
$t\in\mathcal{T}\left(  M\right)  $ et fixons une composante connexe
$\gamma_{t}$ de $M_{t}$. D'apr\`{e}s~\cite[th. 3.3]{HaR1977}, il existe dans
$\gamma_{t}$ un compact $A_{t}$ dont la mesure 1-dimensionnelle de Hausdorff
est nulle et tel que $\mathcal{X}_{t}\cup\gamma_{t}$ est au voisinage de
chacun des points de $\gamma_{t}\backslash A_{t}$ une sous-vari\'{e}t\'{e} de
classe $C^{r}$ de $\mathbb{C}_{t}^{2}$. Soit $p\in\gamma_{t}\backslash A_{t}$.
Puisque $N$ est un anneau l\'{e}vi-plat, $N_{t}$ est au voisinage de $p$ une
r\'{e}union de courbes complexes lisses. $M$ \'{e}tant plong\'{e}e dans $N$ et
$t$ \'{e}tant r\'{e}gulier pour $M$, $M_{t}$ est, au voisinage de $p$, une
sous-variet\'{e} de l'une de ces courbes complexes. Notons la $%
C%
_{p}$ et posons $%
C%
_{p}^{+}=%
C%
_{p}\cap N^{+}$. Alors, dans un voisinage ouvert relativement compact
suffisamment petit de $p$ dans $N$, les courbes complexes $%
C%
_{p}^{+}$ et $\mathcal{X}_{t}$ co\"{\i}ncident car bord\'{e}es au sens des
courants par la m\^{e}me courbe r\'{e}elle $\gamma_{t}$. Puisque
$\mathcal{X}_{t}$ est bord\'{e}e par $M_{t}$, on en d\'{e}duit par un principe
d'unicit\'{e} de prolongement analytique que $\mathcal{X}_{t}$ contient
$N_{t}^{+}$.

Supposons maintenant que $t\notin\mathcal{T}\left(  M\right)  $. On se donne
alors un $\left(  n+1,0\right)  $-cycle $M^{t}$ et un ouvert $N^{t}$ de $N$
comme dans le lemme~\ref{L/ chgt var}. $N^{o}\backslash M^{t}$ a deux
composantes connexes~; on note $N^{t,+}$ celle qui ne rencontre pas $M$.
Appliquant ce qui pr\'{e}c\`{e}de au cycle $M^{t}$ pour lequel $t$ est
r\'{e}gulier, on obtient que le support de $X_{t}^{t}$ contient $\left(
N^{t,+}\right)  _{t}$. Le support $\mathcal{X}_{t}$ du courant $X_{t}=\left[
\left(  X^{t}\right)  _{t}\right]  +\left[  \left(  N^{t}\right)  _{t}\right]
$ contient donc $\left(  N^{t,+}\right)  _{t}\cup\left(  N^{t}\right)  _{t}$,
c'est \`{a} dire $N_{t}^{+}$\smallskip.
\end{proof}

Pour d\'{e}montrer le lemme ci-dessous, on utilise les formules de Harvey et
Lawson~\cite{HaR-LaB1975} revisit\'{e}es par~\cite{DiT1998a}.

\begin{lemma}
\label{L/ reg de Xtilde}$\mathcal{X}$ est un sous-ensemble r\'{e}el analytique
connexe de $\mathbb{E}^{n,2}\backslash M$ de dimension pure $n+2$, ses
sections sont des courbes complexes et sa partie singuli\`{e}re $\mathcal{X}%
^{s}$ est de $\mathcal{H}^{n}$-mesure finie.
\end{lemma}

\begin{proof}
Pour \'{e}tablir ce lemme, il suffit de fixer un param\`{e}tre $t_{\ast}$ et
de prouver que $\mathcal{X}$ a les propri\'{e}t\'{e}s requises au dessus d'un
voisinage de $t_{\ast}$. Etant donn\'{e} que $\mathcal{X}$ contient $N^{+}$ et
que pour tous les param\`{e}tres $t$, $M^{t}\cup M$ borde un ouvert de $N^{+}%
$, on peut supposer sans perte de g\'{e}n\'{e}ralit\'{e} que $M$ est
r\'{e}elle analytique et que $t_{\ast}$ est r\'{e}gulier pour $M$~; autrement
dit, il suffit de prouver le lemme pour $\widetilde{\mathcal{X}}$.

Soient alors $p_{\ast}=\left(  t_{\ast},\omega_{\ast}\right)  =\left(
t_{\ast},\zeta_{\ast},w_{\ast}\right)  \in\mathcal{X}_{t_{\ast}}$ et $\pi$ une
projection de $\mathbb{C}^{2}$ sur une droite complexe de $\mathbb{C}^{2}$
telle que lorsque $t$ est dans l'adh\'{e}rence d'un voisinage ouvert
$\Theta_{\ast}$ suffisamment petit de $t_{\ast}$, $\omega_{\ast}$
n'appartienne pas \`{a} l'image par $\pi_{t}:\mathbb{C}_{t}^{2}\ni\left(
t,z\right)  \mapsto\pi\left(  z\right)  $ de $\widetilde{M}_{t}=\pi
_{\mathbb{C}}\left(  M_{t}\right)  $. Quitte \`{a} changer de coordonn\'{e}es
dans $\mathbb{C}^{2}$, on suppose que $\pi\left(  z_{1},z_{2}\right)  \equiv
z_{1}$ et quitte \`{a} diminuer $\Theta_{\ast}$, on suppose que $\pi_{t}$
projette la courbe r\'{e}elle analytique lisse $M_{t}$ sur une courbe
r\'{e}elle compacte~; l'ouvert $\Gamma^{t}=\mathbb{C}\backslash\pi_{t}\left(
M_{t}\right)  $ n'a alors qu'un nombre fini de composantes connexes qu'on note
$\Gamma_{0}^{t},...,\Gamma_{m_{t}}^{t}$, $\Gamma_{0}^{t}$ \'{e}tant celle qui
est non born\'{e}e. On pose $\Gamma_{\ast}=\ \underset{t\in\Theta_{\ast}}%
{\cup}\left\{  t\right\}  \times\Gamma^{t}$~; $\Gamma_{\ast}$ est un ouvert de
$\mathbb{E}^{n,1}$ qui contient $\left(  t_{\ast},\zeta_{\ast}\right)  $.

On sait d'apr\`{e}s les travaux cit\'{e}s pr\'{e}c\'{e}demment que
\[
X_{t}\left\vert _{\mathbb{C}_{t}^{2}\backslash\pi_{t}^{-1}\left(  \Gamma
^{t}\right)  }\right.  =\frac{i}{\pi}\partial\overline{\partial}\ln\left\vert
R_{t}\right\vert
\]
o\`{u} $R_{t}$ est une fonction qui sur $\Gamma^{t}\times\mathbb{C}$ est
holomorphe par rapport \`{a} la premi\`{e}re variable et rationnelle par
rapport \`{a} la seconde sur chacun des ouverts $\pi_{t}^{-1}\left(
\Gamma_{j}^{t}\right)  $ ($0\leqslant j\leqslant m_{t}$). Pour $\left(
t,\zeta\right)  \in\Gamma_{\ast}$, $R_{t}\left(  \zeta,.\right)  $ est
donn\'{e}e explicitement pour $\left\vert w\right\vert >>1$ par les formules
suivantes~:%
\begin{equation}
R_{t}\left(  \zeta,w\right)  =w^{S_{t,0}\left(  \zeta\right)  }\exp\left(
-\sum\limits_{k\in\mathbb{N}^{\ast}}\frac{S_{t,k}\left(  \zeta\right)  }%
{k}\frac{1}{w^{k}}\right)  ,~S_{t,k}\left(  \zeta\right)  =\frac{1}{2\pi
i}\int_{z\in M_{t}}z_{2}^{k}\frac{dz_{1}}{z_{1}-\zeta}. \label{Def Rt}%
\end{equation}
Cette formule indique non seulement que les fonctions $S_{k}:\left(
t,\zeta\right)  \mapsto S_{t,k}\left(  \zeta\right)  $ sont r\'{e}elles
analytiques sur $\Gamma_{\ast}$ mais aussi, en utilisant comme dans le
lemme~\ref{L/reecrire mom} la formule de Stokes, la forme de Poincar\'{e} et
des d\'{e}formations ad\'{e}quates de $M$ dans $N$, que pour chaque composante
connexe $\gamma$ de $\Gamma_{\ast}$, les fonctions $S_{k}\left\vert _{\gamma
}\right.  $ se prolongent en fonction r\'{e}elles analytiques au voisinage de
$\overline{\gamma}$. Le fait que les sections r\'{e}guli\`{e}res de $M$
satisfont \`{a} la condition des moments entra\^{\i}ne que $S_{t,k}=0$ sur
$\Gamma_{0}^{t}$ quelque soit $k\in\mathbb{N}$ et $t\in\Theta_{\ast}$. Notons
que $M_{t}$ \'{e}tant une r\'{e}union de courbes r\'{e}elles lorsque
$t\in\Theta_{\ast}$, $S_{0}\left(  t,\zeta\right)  $ est un entier puisque
c'est l'indice du cycle $\left[  M\right]  _{t}$ par rapport au point $\left(
t,\zeta\right)  $~; si $t\in\Theta_{\ast}$ et $1\leqslant j\leqslant m_{t} $,
on note $s_{t,j}$ la valeur que prend $S_{0}$ sur $\Gamma_{j}^{t}$.

Puisque nous savons a priori d'apr\`{e}s~\cite{HaR-LaB1975} (voir aussi
\cite{HaR1977}) que les fonctions $R_{t}\left(  \zeta,.\right)  $ sont des
fractions rationnelles en $w$, pour tout $k\in\mathbb{N}$, $S_{k}\left(
t,\zeta\right)  $ est la diff\'{e}rence entre la somme des puissances
$k$-\`{e}mes des racines de $R_{t}\left(  \zeta,.\right)  $
r\'{e}p\'{e}t\'{e}s avec leur multiplicit\'{e} et la m\^{e}me somme mais
portant sur les p\^{o}les de $R_{t}\left(  \zeta,.\right)  $. Notons $R$ la
fonction $R:(t,\zeta,w)\mapsto R_{t}\left(  \zeta,w\right)  $. En ayant \`{a}
l'esprit l'ajout d'un param\`{e}tre r\'{e}el, il d\'{e}coule des arguments de
Harvey et Lawson dans \cite[th. 4.6]{HaR-LaB1975} et \cite[lemme~3.21 \&
3.19]{HaR1977} que sur chaque composante connexe $\gamma$ de $\Gamma_{\ast}$,
$R\left(  t,\zeta,w\right)  =\frac{P\left(  t,\zeta,w\right)  }{Q\left(
t,\zeta,w\right)  }$ o\`{u} $P$ et $Q$ sont des polyn\^{o}mes unitaires et
irr\'{e}ductibles de l'anneau $CR^{\omega}\left(  \gamma\right)  \left[
w\right]  $.

Prouvons maintenant que pour toute composante connexe $\gamma$ de
$\Gamma_{\ast}$, $\widetilde{\mathcal{X}}$ co\"{\i}ncide au dessus de $\gamma$
via la projection $\pi_{1}:\mathbb{E}^{n,2}\ni\left(  t,z\right)
\mapsto\left(  t,z_{1}\right)  $ avec l'ensemble des points $\left(
t,\zeta,w\right)  $ de $\mathbb{E}^{n,2}\backslash M$ tels que $\left(
t,\zeta\right)  \in\gamma$ et $P\left(  t,\zeta,w\right)  Q\left(
t,\zeta,w\right)  =0$. Puisque $P$ et $Q$ sont premiers entre eux dans
$CR^{\omega}\left(  \gamma\right)  \left[  w\right]  $, $P\left(
t,\zeta,.\right)  $ et $Q\left(  t,\zeta,.\right)  $ sont premiers entre eux
pour $\left(  t,\zeta\right)  $ variant dans un ouvert dense $\gamma
^{\,\prime}$ de $\gamma$~; lorsque $\left(  t,\zeta\right)  \in\gamma
^{\,\prime}$, $\mathcal{X}_{t}$ contient l'ensemble des points $\left(
t,\zeta,w\right)  $ de $\mathbb{E}^{n,2}$ tels que $\left(  t,\zeta\right)
\in\gamma^{\,\prime}$ et $P\left(  t,\zeta,w\right)  Q\left(  t,\zeta
,w\right)  =0$. Soit $\left(  t_{\infty},\zeta_{\infty}\right)  \in\gamma$ et
$w_{\infty}\in\mathbb{C}$ tel que $p_{\infty}=\left(  t_{\infty},\zeta
_{\infty},w_{\infty}\right)  \in\mathbb{E}^{n,2}\backslash M$ et annule $PQ$.
Soit $\left(  p_{\nu}\right)  _{\nu\in\mathbb{N}^{\ast}}$ une suite de points
$p_{\nu}=\left(  t_{\nu},\zeta_{\nu},w_{\nu}\right)  $ de $\mathbb{E}^{n,2}$
convergeant vers $p_{\infty}$ telle que $\left(  t_{\nu},\zeta_{\nu}\right)
\in\gamma^{\,\prime}$ et $\left(  t_{\nu},\zeta_{\nu},w_{\nu}\right)
\in\mathcal{X}_{t_{\nu}}$. Puisque $M_{t_{\nu}}$ ($\nu\in\mathbb{N}$) est une
section r\'{e}guli\`{e}re de $M$, $M_{t_{\nu}}$ est bien le support de
$\left[  M\right]  _{t_{\nu}}$ et $M_{t_{\nu}}$ converge vers $M_{t_{\infty}}$
au sens de la m\'{e}trique de Hausdorff. Etant donn\'{e} que $\overline
{\Theta_{\ast}}\subset\mathcal{T}\left(  M\right)  $, les $X_{t_{\nu}}$
($\nu\in\mathbb{N}^{\ast}$) sont de masse uniform\'{e}ment born\'{e}e. On peut
donc d'apr\`{e}s~\cite[th. 4.2]{HaR1977} supposer, modulo une \'{e}ventuelle
extraction, que $X_{t_{\nu}}$ converge au sens des courants vers une 1-chaine
holomorphe $X_{\infty}$ de $\mathbb{C}_{t_{\infty}}^{2}\backslash
M_{t_{\infty}}$ v\'{e}rifiant $dX_{\infty}=\left[  M\right]  _{t_{\infty}%
}=dX_{t_{\infty}}$, ce qui force $X_{\infty}=X_{t_{\infty}}$. Puisque
$p_{\infty}\notin M$, on en d\'{e}duit que $p_{\infty}\in\operatorname{Supp}%
X_{t_{\infty}}=\mathcal{X}_{t_{\infty}}$.

Finalement, ceci prouve l'existence d'un voisinage $G$ de $p_{\ast}$ dans
$\mathbb{E}^{n,2}$ tel que que $\widetilde{\mathcal{X}}\cap G$ est l'ensemble
des points $\left(  t,\zeta,w\right)  $ de $G\backslash M$ tels que $\left(
t,\zeta\right)  \in\Gamma_{\ast}$ et $P\left(  t,\zeta,w\right)  Q\left(
t,\zeta,w\right)  =0$. C'est donc un sous-ensemble r\'{e}el analytique de $G$
de dimension $n+2$. Si $V$ est l'image de $G$ par $\pi_{1}$ et si
$D_{1},...,D_{k}$ sont les diviseurs irr\'{e}ductibles deux \`{a} deux
distincts de $PQ$ dans $CR^{\omega}\left(  V\right)  \left[  w\right]  $,
$\widetilde{\mathcal{X}}{}^{s}\cap G$ est contenu dans l'ensemble des points
$\left(  t,\zeta,w\right)  $ de $G$ tels que $\left(  t,\zeta\right)  $ annule
le discriminant $\Delta$ de $D_{1}....D_{k}$. Etant donn\'{e} que $\Delta\in
CR^{\omega}\left(  V\right)  $, l'ensemble des solutions de l'\'{e}quation
$\Delta\left(  t,\zeta\right)  =0$ \`{a} $t$ fix\'{e} est au plus discret et
il s'ensuit que $\widetilde{\mathcal{X}}{}^{s}\cap G\cap\mathcal{X}_{t}$ est
aussi au plus discret. Au voisinage de $p_{\ast}$, $\widetilde{\mathcal{X}}%
{}^{s}$ est donc de dimension au plus $n$ et de $\mathcal{H}^{n}$-volume fini.
\end{proof}

\begin{lemma}
\label{L/ dX = M}$\mathcal{X}$ est connexe, orientable et la formule de Stokes
$d\left[  \mathcal{X}\right]  =\pm\left[  M\right]  $ est valide. De plus, il
existe un compact $A$ de $M$ tel que $\mathcal{H}^{n+1}\left(  A\right)  =0$
et $%
\overline{\mathcal{X}}%
\backslash A$ est localement une vari\'{e}t\'{e} \`{a} bord de classe $C^{r}$
de bord $M$.
\end{lemma}

\begin{proof}
$\mathcal{X}$ est connexe par arcs car $\mathcal{X}$ contient $N^{+}$ qui
pr\`{e}s de $M$ est de classe $C^{1}$ et bord\'{e}e par $M$ qui est
elle-m\^{e}me connexe par arcs. Pour chaque param\`{e}tre $t$ r\'{e}gulier
pour $M$, il existe d'apr\`{e}s~\cite{HaR-LaB1975} un compact $A_{t}$ tel que
$\mathcal{H}^{1}\left(  A_{t}\right)  =0$ et tel que $\overline{\mathcal{X}%
_{t}}\backslash A_{t}$ est localement une vari\'{e}t\'{e} \`{a} bord de classe
$C^{r}$ de bord $M_{t}$. Puisque $M$ est de classe $C^{1}$, le lemme de Sard
donne que $\mathcal{H}^{n}\left(  \pi\left(  M\right)  \backslash
\mathcal{T}\left(  M\right)  \right)  =0$. Il en r\'{e}sulte que la
r\'{e}union $A$ de $M\backslash\pi^{-1}\left(  \mathcal{T}\left(  M\right)
\right)  $ et de$\underset{t\in\mathcal{T}\left(  M\right)  }{\cup}A_{t}$ a
les propri\'{e}t\'{e}s requises dans l'\'{e}nonc\'{e} du lemme.

En tant que courbe complexe, chaque section $\mathcal{X}_{t}$ de $\mathcal{X}$
poss\`{e}de une orientation canonique qui est aussi celle de $N_{t}$ puisque
$N^{+}\subset\mathcal{X}$. Lorsque $t_{\ast}$ est un param\`{e}tre
r\'{e}gulier de $M$, il est possible d'orienter de fa\c{c}on coh\'{e}rente les
sections $M_{t}$ voisines de $M_{t}$, ce qui implique qu'il existe
$\varepsilon\in\left\{  -1,+1\right\}  $ tel que $d\left[  \mathcal{X}%
_{t}\right]  =\varepsilon\left[  M_{t}\right]  $ lorsque $t$ est voisin de
$t_{\ast}$. Lorsque $t_{\ast}$ n'est pas un param\`{e}tre r\'{e}gulier de $M$,
la m\^{e}me conclusion s'applique \`{a} $M^{t_{\ast}}$ et $\mathcal{X}%
^{t_{\ast}}$ et, puisque $N^{+}\subset\mathcal{X}$, \`{a} $M$ et $\mathcal{X}%
$. $M$ \'{e}tant connexe, on en d\'{e}duit qu'il existe $\varepsilon$ dans
$\left\{  -1,1\right\}  $ tel que $d\left[  \mathcal{X}_{t}\right]
=\varepsilon\left[  M_{t}\right]  $.$\mathcal{X}$ est donc orientable et
$d\left[  \mathcal{X}\right]  =\varepsilon\left[  M\right]  $.
\end{proof}

La synth\`{e}se des lemmes pr\'{e}c\'{e}dents \'{e}tablit l'existence d'un
ensemble $\mathcal{X}$ r\'{e}solvant le probl\`{e}me pos\'{e} dans le
th\'{e}or\`{e}me~\ref{T/bord CxR A}. La preuve de ce th\'{e}or\`{e}me
s'ach\`{e}ve avec le lemme suivant~:

\begin{lemma}
\label{L/ Unicite}Il existe au plus un ensemble r\'{e}el analytique
orient\'{e} de dimension $n+2$ remplissant les conditions (1) et (2) du
th\'{e}or\`{e}me~\ref{T/bord CxR A}.
\end{lemma}

\begin{proof}
Supposons que $%
Z%
$ soit solution du probl\`{e}me pos\'{e} dans le
th\'{e}or\`{e}me~\ref{T/bord CxR A}. Soit $t\in\mathcal{T}\left(  M\right)  $.
$\left[
Z%
\right]  _{t}$ est une 1-cha\^{\i}ne holomorphe de masse finie de
$\mathbb{C}_{t}^{2}$ dont le support est contenu dans $%
Z%
_{t}$ et donc dans $\mathbb{C}_{t}^{2}\backslash M_{t}$. Comme $\left(
-1\right)  ^{n}d\left[
Z%
\right]  _{t}=\left(  d\left[
Z%
\right]  \right)  _{t}=\left[  M\right]  _{t}=\left[  M_{t}\right]  $, on en
d\'{e}duit que $\left[
Z%
\right]  _{t}=\pm\left[  \mathcal{X}_{t}\right]  $ et donc que $%
Z%
$ contient $\mathcal{X}_{t}$. Ainsi, $%
Z%
$ contient $\widetilde{\mathcal{X}}$ et donc $\mathcal{X}$ puisque $%
Z%
$ est ferm\'{e} dans $M^{c}=\mathbb{E}^{n,2}\backslash M$. Notons $%
S%
$ la r\'{e}union des ensembles singuliers de $%
Z%
$ et $\mathcal{X}$. $\mathcal{X}\backslash%
S%
$ et $%
Z%
\backslash%
S%
$ sont donc deux vari\'{e}t\'{e}s de m\^{e}me dimension $n+2$~; comme
$\mathcal{X}\backslash%
S%
\subset%
Z%
\backslash%
S%
$, $\mathcal{X}\backslash%
S%
$ est un ouvert de $%
Z%
\backslash%
S%
$~; c'est aussi un ferm\'{e} de $%
Z%
\backslash%
S%
$ car $\mathcal{X}$ est ferm\'{e} dans $M^{c}$ et $%
Z%
$ contenu dans $M^{c}$. Puisque $%
Z%
$ et $\mathcal{X}$ sont orientables, $\mathcal{H}^{n+1}\left(
S%
\right)  =0$ de sorte que $%
Z%
\backslash%
S%
$ est connexe. Par cons\'{e}quent $\mathcal{X}\backslash%
S%
=%
Z%
\backslash%
S%
$ et par densit\'{e}, $\mathcal{X}=%
Z%
$.\medskip
\end{proof}

\begin{lemma}
\label{L/ coherence}$\mathcal{X}$ est un ensemble r\'{e}el analytique
coh\'{e}rent au sens de \cite{CaH1957}.
\end{lemma}

\begin{proof}
D'apr\`{e}s~\cite[prop. 12 p. 94]{CaH1957}, il s'agit de construire pour tout
point $p_{\ast}$ de $X$, un sous-ensemble analytique complexe $Z$ d'un
voisinage de $p_{\ast}$ dans le complexifi\'{e} $\mathbb{C\mathbb{E}}%
^{n,2}=\mathbb{C}^{n}\times\left(  \mathbb{C}\times\mathbb{C}\right)  ^{2}$ de
$\mathbb{E}^{n,2}$ tel que pour tout point $p$ voisin de $p_{\ast}$ dans $X$,
le germe $Z_{\underline{\underline{p}}}$ de $Z$ en $p$ soit le complexifi\'{e}
du germe $X_{\underline{\underline{p}}}$ de $X$ en $p$.\smallskip

Notons tout d'abord que $M$ \'{e}tant une sous-vari\'{e}t\'{e} r\'{e}elle
analytique de $\mathbb{\mathbb{E}}^{n,2}$ de dimension r\'{e}elle $n+1$, elle
est coh\'{e}rente~; sa complexifi\'{e}e $\widetilde{M}$ est une
sous-vari\'{e}t\'{e} complexe d'un voisinage ouvert de $M$ dans
$\mathbb{C\mathbb{E}}^{n,2}$ de dimension complexe $n+1$. $N$ n'est pas une
sous-vari\'{e}t\'{e} mais est localement une r\'{e}union de
sous-vari\'{e}t\'{e}s de $\mathbb{\mathbb{E}}^{n,2}$ et il est clair que $N$
est coh\'{e}rente~; on note $\widetilde{N}$ le complexifi\'{e} de $N$~; il est
de dimension complexe $n+2$.

Dans un premier temps, on construit un objet global candidat \`{a} \^{e}tre le
complexifi\'{e} de $X$. On complexifie les variables naturelles $t$,
$z_{1}=x_{1}+iy_{1}$ et $z_{2}=x_{2}+iy_{2}$ de $\mathbb{E}^{n,2}$ en posant
$\tau=t+is$, $\xi_{j}=x_{j}+ix_{j}^{\prime}$, $\eta_{j}=y_{j}+iy_{j}^{\prime}%
$, $\zeta_{j}=\xi_{j}+i\eta_{j}$, $j=1,2$. Lorsque $\left(  s,x^{\prime
},y^{\prime}\right)  \in\mathbb{R}^{n}\times\mathbb{R}^{2}\times\mathbb{R}%
^{2}$ on pose
\[
\mathbb{\mathbb{E}}_{\left(  s,x^{\prime},y^{\prime}\right)  }^{n,2}=\left\{
\left(  \tau,\xi_{1},\eta_{1},\xi_{2},\eta_{2}\right)  \in\mathbb{C\mathbb{E}%
}^{n,2}~;~\operatorname{Im}\left(  \tau,\xi_{1},\eta_{1},\xi_{2},\eta
_{2}\right)  =\left(  s,x^{\prime},y^{\prime}\right)  \right\}  ~;
\]
si $\left(  s,x^{\prime},y^{\prime}\right)  $ est de norme infini assez
petite
\[
\widetilde{N}_{\left(  s,x^{\prime},y^{\prime}\right)  }=\widetilde{N}%
\cap\mathbb{\mathbb{E}}_{\left(  s,x^{\prime},y^{\prime}\right)  }^{n,2}\text{
}%
\]
est par construction un anneau l\'{e}vi-plat dans lequel est plong\'{e}
$\widetilde{M}_{\left(  s,x^{\prime},y^{\prime}\right)  }=\widetilde{M}%
\cap\mathbb{\mathbb{E}}_{\left(  s,x^{\prime},y^{\prime}\right)  }^{n,2}$. Il
s'ensuit qu'hormis sa conclusion finale, on peut appliquer le
th\'{e}or\`{e}me~\ref{T/bord CxR A} pour obtenir un sous-ensemble r\'{e}el
analytique $Y_{\left(  s,x^{\prime},y^{\prime}\right)  }$ de
$\mathbb{\mathbb{E}}_{\left(  s,x^{\prime},y^{\prime}\right)  }^{n,2}$
feuillet\'{e} par des courbes complexes $Y_{\left(  t,s,x^{\prime},y^{\prime
}\right)  }=Y_{\left(  s,x^{\prime},y^{\prime}\right)  }\cap\left\{
\operatorname{Re}\tau=t\right\}  $. Puisque $X$ co\"{\i}ncide avec $N$ dans un
ouvert de $N\backslash M$, $Y$ co\"{\i}ncide avec $\widetilde{N}$ dans un
ouvert de $(\widetilde{N}\backslash\widetilde{M})$ et y est donc un
sous-ensemble analytique complexe de dimension complexe $n+2$.

Fixons maintenant un point $p_{\ast}=\left(  t_{\ast},z_{\ast}\right)  $ de
$X$. Puisque $X$ co\"{\i}ncide avec $N$ dans un ouvert de $N\backslash M$, on
peut supposer sans perte de g\'{e}n\'{e}ralit\'{e} que $t_{\ast}$ est
r\'{e}gulier pour $M$. Pour chaque point $p$ de $X_{t_{\ast}}$, on peut alors
trouver un voisinage de $p$ dans $\mathbb{E}^{n,2}$ et des coordonn\'{e}es
$z^{p}=\left(  z_{1}^{p},z_{2}^{p}\right)  $ de $\mathbb{C}^{2}$ centr\'{e}es
en $p$ et telles que $p$ est au-dessus de l'une des composantes connexes de
$\left(  \pi_{t_{\ast}}^{p}\right)  ^{-1}\left(  \pi_{t_{\ast}}^{p}\left(
M\right)  \right)  $ o\`{u} $\pi_{t_{\ast}}^{p}$ est la projection $\left(
t_{\ast},z^{p}\right)  \mapsto z_{1}^{p}$. On sait alors que $X$ est donn\'{e}
au voisinage de $p$ dans $\mathbb{E}^{n,2}$ par une \'{e}quation de la forme
$F^{p}\left(  t,z^{p}\right)  =0$ o\`{u} $F^{p}$ est un polyn\^{o}me de
Weiertrass en $z_{2}^{p}$. Posons $U^{p}=\operatorname{Re}F^{p}$ et
$V^{p}=\operatorname{Im}F^{p}$~; $X$ est d\'{e}fini par les \'{e}quations
$U^{p}\left(  t,z\right)  =V^{p}\left(  t,z\right)  =0$. Consid\'{e}rons dans
un voisinage ouvert de $p$ dans $\mathbb{C\mathbb{E}}^{n,2}$ le sous-ensemble
analytique complexe $Z^{p}$ d\'{e}fini par les \'{e}quations%
\[
U^{p}\left(  \tau,\zeta^{p}\right)  =V^{p}\left(  \tau,\zeta^{p}\right)  =0.
\]
o\`{u} $\zeta^{p}=\left(  \zeta_{1}^{p},\zeta_{2}^{p}\right)  $, $\zeta
_{1}^{p}=\xi_{1}^{p}+i\eta_{1}^{p}$ et $\zeta_{2}=\xi_{2}^{p}+i\eta_{2}^{p}$
complexifient les variables $z_{1}^{p}=x_{1}^{p}+iy_{1}^{p}$ et $z_{2}%
^{p}=x_{2}^{p}+iy_{2}^{p}$~; on a donc%
\[
\left\{
\begin{array}
[c]{c}%
U^{p}\left(  \tau,\zeta^{p}\right)  =\frac{1}{2}\left[  F^{p}\left(  \tau
-\tau_{\ast},\zeta^{p}\right)  +\overline{F^{p}}\left(  \tau-\tau_{\ast}%
,\xi_{1}-i\eta_{1},\xi_{2}-i\eta_{2}\right)  \right] \\
V^{p}\left(  \tau,\zeta^{p}\right)  =\frac{1}{2i}\left[  F^{p}\left(
\tau-\tau_{\ast},\zeta^{p}\right)  -\overline{F^{p}}\left(  \tau-\tau_{\ast
},\xi_{1}-i\eta_{1},\xi_{2}-i\eta_{2}\right)  \right]
\end{array}
\right.  .
\]
Par construction les ensembles $Z^{p}$, $p$ variant dans $X_{t_{\ast}}$, se
recollent le long de $X_{t_{\ast}}$. Si $p$ est un point r\'{e}gulier de
$X_{t_{\ast}}$, $Z^{p}$ est aussi r\'{e}guli\`{e}re en $p$ et est clairement
une complexification de $X_{t_{\ast}}$ au voisinage de $p$. L'unicit\'{e} de
cette complexification et le fait que les singularit\'{e}s de $X_{t_{\ast}}$
sont isol\'{e}es fait que les $Z^{p}$, $p$ variant dans $X_{t_{\ast}}$, se
recollent pr\`{e}s de $X_{t_{\ast}}$. Notons $Z_{t_{\ast}}$ la r\'{e}union des
$Z^{p}$ lorsque $p$ varie dans $X_{t_{\ast}}$. Alors $Z_{t_{\ast}}$ est un
sous-ensemble analytique complexe d'un ouvert de $\mathbb{CE}^{n,2}%
\backslash\widetilde{M}$ et par construction lorsque $t$ est suffisamment
voisin de $t_{\ast}$, la trace sur $\mathbb{E}_{\left(  0,0,0\right)  }%
^{n,2}\cap\left\{  \operatorname{Re}\tau=t\right\}  $ est $X_{t}$. Puisque $X$
co\"{\i}ncide avec $N$ dans un ouvert de $N\backslash M$ et que $N$ est
localement une r\'{e}union finie de sous-vari\'{e}t\'{e}s de $\mathbb{E}%
^{n,2}$, $Z$ co\"{\i}ncide avec $\widetilde{N}$ dans un ouvert de
$(\widetilde{N}\backslash\widetilde{M})\cap\left\{  \left\Vert t-t_{\ast
}\right\Vert <\varepsilon\right\}  $ du moment que $\varepsilon$ est assez
petit. On en d\'{e}duit que pour chaque triplet $\left(  s,x^{\prime
},y^{\prime}\right)  $ de r\'{e}els assez petits, $Z\cap\mathbb{\mathbb{E}%
}_{\left(  s,x^{\prime},y^{\prime}\right)  }^{n,2}$ est feuillet\'{e} par les
courbes complexes que sont les $Y_{\left(  t,s,x^{\prime},y^{\prime}\right)
}$, $t$ variant dans un voisinage de $t_{\ast}$. $Z$ a donc une structure de
sous-ensemble analytique complexe feuillet\'{e} par des courbes complexes
param\`{e}tr\'{e}es par $\left(  t,s,x^{\prime},y^{\prime}\right)  $ variant
dans un voisinage de $\left(  t_{\ast},0,0,0\right)  $ dans $\mathbb{R}%
^{n}\times\mathbb{R}^{n}\times\mathbb{R\times R}$. Il est donc de dimension
r\'{e}elle $2n+4$ et donc de dimension complexe $n+2$. Il s'ensuit que $Z$ est
au voisinage de $X_{t_{\ast}}$ un complexifi\'{e}. En particulier, $X$ est
coh\'{e}rent en $p_{\ast}$.
\end{proof}

\subsection{Preuve des propositions~\ref{P/ harmo},~\ref{P/ Green} et du
th\'{e}or\`{e}me~\ref{T/ harmo}}

\subsubsection{Preuve des propositions~\ref{P/ harmo},~\ref{P/ Green}}

Notons $\mu$ la mesure sur $\mathbb{C}^{n}$ que $\alpha$ d\'{e}finit par
int\'{e}gration sur $\gamma$. Les moments de $\mu$ sont alors ceux de $\alpha$
et d'apr\`{e}s \cite{HeG1995} et \cite{DiT1998b}, il existe dans
$\mathbb{C}^{n}\backslash\gamma$ une courbe complexe $\mathcal{X}$ de masse
finie bordant $\gamma$ au sens des courants et une $\left(  1,0\right)
$-forme holomorphe $\widetilde{\alpha}$ sur $\mathcal{X}$ telles que $d\left(
\left[  \mathcal{X}\right]  \wedge\widetilde{\alpha}\right)  =\mu$. Notons
$\operatorname*{Reg}\overline{\mathcal{X}}$ la partie r\'{e}guli\`{e}re de
$\overline{\mathcal{X}}$, c'est \`{a} dire l'ensemble des points r\'{e}guliers
de la courbe complexe $\mathcal{X}$ et des points de $\gamma$ au voisinage
desquels $\overline{\mathcal{X}}$ est une vari\'{e}t\'{e} \`{a} bord de bord
$\gamma$. Puisque $\gamma$ est lisse, on sait que $\sigma=\gamma
\backslash\operatorname*{Reg}\overline{\mathcal{X}}$ est de mesure
1-dimensionnelle nulle. La forme $\widetilde{\alpha}+\overline{\widetilde
{\alpha}}$ est ferm\'{e}e et la fonction multivalu\'{e}e qui vaut $u$ sur
$\gamma\cap\operatorname*{Reg}\overline{\mathcal{X}}$ et qui est obtenue par
int\'{e}gration de $\widetilde{\alpha}+\overline{\widetilde{\alpha}}$ le long
des chemins de $\operatorname*{Reg}\overline{\mathcal{X}}$ est une solution au
probl\`{e}me pos\'{e} dans la proposition~\ref{P/ harmo}.\medskip

Dans le cas d'une courbe de $\mathbb{C}^{2}$, la preuve de la
proposition~\ref{P/ Green} est contenue dans celle du
th\'{e}or\`{e}me~\ref{T/ harmo}. L'adaptation au cas d'une courbe de
$\mathbb{C}^{n}$ ne pr\'{e}sente d'autre difficult\'{e} que celle d'avoir
\`{a} utiliser des notations plus lourdes et nous en omettons la
preuve.\medskip

\subsubsection{Preuve du th\'{e}or\`{e}me~\ref{T/ harmo}}

Pla\c{c}ons nous maintenant dans la situation du
th\'{e}or\`{e}me~\ref{T/ harmo}. Les conclusions et notations du
th\'{e}or\`{e}me~\ref{T/bord CxR A} ainsi que les lemmes qui ont servi \`{a}
sa preuve sont donc en vigueur. Puisque $N$ est un anneau l\'{e}vi-plat, il
existe pour chaque point $p$ de $M$ un voisinage ouvert $G_{p}$ de $p$ dans
$\mathbb{E}^{n,2}$ tel que $M\cap G_{p}$ est contenu dans une seule des
sous-vari\'{e}t\'{e}s $\widetilde{N}_{p}$ dont $N\cap G_{p}$ est la
r\'{e}union. On peut alors d\'{e}finir sur $\widetilde{N}=\underset{p\in
M}{\cup}\widetilde{N}_{p}$ une $\left(  1,0\right)  $-forme holomorphe
$\alpha$ en posant $\alpha\left\vert _{\widetilde{N}_{p}}\right.
=\partial\left(  u\left\vert _{\widetilde{N}_{p}\cap\mathbb{C}_{t}^{2}%
}\right.  \right)  $, $p$ variant dans $M$. En utilisant les arguments de la
preuve du th\'{e}or\`{e}me~\ref{T/bord CxR A}, on peut obtenir que chaque
section $\alpha\left(  t,.\right)  $ de $\alpha$ v\'{e}rifie la condition des
moments (\ref{Mom forme sur courbe}) et donc appliquer la
proposition~\ref{P/ harmo} aux sections $M_{t}$ de $M$. Cependant, plut\^{o}t
que de prouver que les cha\^{\i}nes holomorphes ainsi produites sont des
courants d'int\'{e}gration sur des courbes complexes et d'\'{e}tablir ensuite
que les fonctions harmoniques multivalu\'{e}es obtenues sont univalu\'{e}es et
d\'{e}pendent analytiquement de $t$, nous utilisons l'ensemble $\mathcal{X}$
donn\'{e} par le th\'{e}or\`{e}me~\ref{T/bord CxR A} comme solution du
probl\`{e}me de bord pos\'{e} pour la famille $\left(  M_{t}\right)  $. Le
prolongement escompt\'{e} pour $u$ s'explicite alors par le biais de formules,
classiques dans le cas lisse, faisant intervenir des familles de fonctions de
Green associ\'{e}es \`{a} des familles de courbes complexes \'{e}ventuellement
singuli\`{e}res. Ces fonctions de Green sont obtenues dans la
proposition~\ref{P/ fction de Green} par dualit\'{e} et la r\'{e}solution du
$\overline{\partial}$ \'{e}tablie dans la proposition~\ref{P/ resdbar}; Pour
mettre en \oe uvre cette approche, nous introduisons quelques
notations.\medskip

$\mathcal{X}$ co\"{\i}ncidant avec $N$ dans un ouvert de $N\backslash M$, on
obtient en r\'{e}unissant $\mathcal{X}$ avec un voisinage ouvert ad\'{e}quat
de $M$ dans $N$, un ensemble r\'{e}el analytique $\mathcal{Y}$ dans lequel
$\mathcal{X}$ est relativement compact et bord\'{e} au sens des courants par
$\pm\left[  M^{\prime}\right]  $ o\`{u} $M^{\prime}$ est une
sous-vari\'{e}t\'{e} plong\'{e}e dans $N$ ayant la m\^{e}me orientation que
$M$. Puisque d'apr\`{e}s le$\;$th\'{e}or\`{e}me~\ref{T/bord CxR A}
$\mathcal{X}$ est un ensemble r\'{e}el analytique coh\'{e}rent, $\mathcal{Y}$
l'est aussi et \`{a} ce titre admet un complexifi\'{e} global. En prenant la
trace de ce complexifi\'{e} sur $\mathbb{C}^{n}\times\mathbb{C}^{\mathbf{2}}$,
on obtient un sous-ensemble r\'{e}el analytique coh\'{e}rent $\widetilde
{\mathcal{Y}}$ d'un ouvert $\mathbb{C}^{n}\times\mathbb{C}^{\mathbf{2}}$ dont
la trace sur $\mathbb{E}^{n,2}$ est $\mathcal{Y}$. Le fait que $\mathcal{Y}$
soit un anneau l\'{e}vi-plat pr\`{e}s de $M$ entra\^{\i}ne que $\widetilde
{\mathcal{Y}}$ est en fait un sous-ensemble analytique complexe de
$\mathbb{C}^{n}\times\mathbb{C}^{\mathbf{2}}$ qui est pr\'{e}cis\'{e}ment le
complexifi\'{e} de $\mathcal{Y}$ par rapport aux $n$ premi\`{e}res variables
r\'{e}elles de $\mathbb{E}^{n,2}$. S\'{e}lectionnons pour $\widetilde
{\mathcal{Y}}$ un voisinage de Stein $\Omega$ dans $\mathbb{C}^{n}%
\times\mathbb{C}^{\mathbf{2}}$ et $H$ une fonction holomorphe sur $\Omega$
telle que $\widetilde{\mathcal{Y}}=\left\{  H=0\right\}  $~; on a donc
$\mathcal{Y}=\left\{  F=0\right\}  $ o\`{u} $F=H\left\vert _{\Omega
\cap\mathbb{E}^{n,2}}\right.  $. Puisque $\overline{\mathcal{X}}$ est un
compact de $\Omega$, on peut aussi se donner dans $\mathbb{C}^{n}%
\times\mathbb{C}^{\mathbf{2}}$ un domaine strictement pseudoconvexe
$\Omega^{\ast}$ de classe $C^{\infty}$ tel que
\[
\mathcal{X}\subset\subset\mathcal{Y}_{0}=\mathcal{Y}\cap\Omega^{\ast}%
\subset\subset\Omega.
\]

On sait d'apr\`{e}s les travaux d'Oka et de Cartan qu'il existe une fonction
sym\'{e}trique $\mathcal{Q}\in\mathcal{O}\left(  \Omega\times\Omega
,\mathbb{C}^{n}\times\mathbb{C}^{2}\right)  $ telle que pour tout $\left(
\tau^{\prime},z^{\prime}\right)  ,\left(  \tau,z\right)  \in\Omega\times
\Omega$,%
\[
H\left(  \tau^{\prime},z^{\prime}\right)  -H\left(  \tau,z\right)
=\left\langle \mathcal{Q}\left(  \tau^{\prime},z^{\prime},\tau,z\right)
,\left(  \tau^{\prime}-\tau,z^{\prime}-z\right)  \right\rangle
\]
o\`{u} on a pos\'{e} $\left\langle v,w\right\rangle =%
{\textstyle\sum\limits_{1\leqslant j\leqslant d}}
v_{j}w_{j}$ lorsque $v,w\in\mathbb{C}^{d}$. D'apr\`{e}s~\cite{HeG-LeJ1984L},
on peut se donner une fonction barri\`{e}re pour $\Omega_{0}$, c'est \`{a}
dire une fonction $\mathcal{Q}^{\left[  0\right]  }$ \`{a} valeurs dans
$\mathbb{C}^{2}$, de classe $C^{\infty}$ au voisinage de $\overline
{\Omega^{\ast}}\times\overline{\Omega^{\ast}}$, holomorphe par rapport \`{a}
sa seconde variable et telle que
\[
\left\langle \mathcal{Q}^{\left[  0\right]  }\left(  \tau^{\prime},z^{\prime
},\tau,z\right)  ,\left(  \tau^{\prime}-\tau,z^{\prime}-z\right)
\right\rangle \neq0
\]
lorsque $\left(  \left(  \tau^{\prime},z^{\prime}\right)  ,\left(
\tau,z\right)  \right)  \in b\Omega^{\ast}\times\Omega^{\ast}$. Dans ce qui
suit, si $f$ est une fonction ou une forme d\'{e}finie sur une partie de
$\mathbb{E}^{n,2}$, on \'{e}crira $f_{t}$ pour $f\left(  t,.\right)  $. Si
$j\in\left\{  1,2\right\}  $, $t\in\mathbb{R}^{n}$, on pose $Q_{j}\left(
t,z^{\prime},z\right)  =\mathcal{Q}_{n+j}\left(  t,z^{\prime},t,z\right)  $
lorsque et $z^{\prime},z\in\Omega^{t}$ et $Q_{j}^{\left[  0\right]  }\left(
t,z^{\prime},z\right)  =\mathcal{Q}_{n+j}^{\left[  0\right]  }\left(
t,z^{\prime},t,z\right)  $ lorsque $z^{\prime},z\in\overline{\Omega^{\ast}}$.
Par construction $Q=\left(  Q_{1},Q_{2}\right)  $ et $Q^{\left[  0\right]
}=\left(  Q_{1}^{\left[  0\right]  },Q_{2}^{\left[  0\right]  }\right)  $
v\'{e}rifient%
\begin{gather*}
\forall t\in\mathbb{R}^{n},~\forall z,z^{\prime}\in\Omega^{t},~F\left(
t,z^{\prime}\right)  -F\left(  t,z\right)  =\left\langle Q\left(  t,z^{\prime
},z\right)  ,z^{\prime}-z\right\rangle \\
\forall t\in\mathbb{R}^{n},~\forall\zeta,z\in b\Omega^{\ast t}\times
\Omega^{\ast t},~\left\langle \mathcal{Q}^{\left[  0\right]  }\left(
\zeta,z\right)  ,\zeta-z\right\rangle \neq0.
\end{gather*}

On d\'{e}finit sur $\operatorname*{Reg}\mathcal{Y}^{t}$ une $\left(
1,0\right)  $-forme $\omega_{t}$ en posant%
\begin{align*}
\omega_{t}  &  =\frac{-dz_{1}}{\partial F_{t}/\partial z_{2}}~\text{sur}%
~\mathcal{Y}^{t,1}=\mathcal{Y}^{t}\cap\left\{  \partial F_{t}/\partial
z_{2}\neq0\right\} \\
\omega_{t}  &  =\frac{+dz_{2}}{\partial F_{t}/\partial z_{1}}~\text{sur}%
~\mathcal{Y}^{t,2}=\mathcal{Y}^{t}\cap\left\{  \partial F_{t}/\partial
z_{1}\neq0\right\}
\end{align*}
puis on pose%
\[
k_{t}\left(  z^{\prime},z\right)  =\det\left[  \frac{\overline{z^{\prime}%
}-\overline{z}}{\left\vert z^{\prime}-z\right\vert ^{2}},Q_{t}\left(
z^{\prime},z\right)  \right]  ~~\&~~K_{t}\left(  z^{\prime},z\right)
=k_{t}\left(  z^{\prime},z\right)  \omega_{t}\left(  z^{\prime}\right)  .
\]

Si $\varepsilon\in\mathbb{R}_{+}^{\ast}$, on pose $\Omega^{\varepsilon
}=\left\{  \left(  \tau,z\right)  \in\Omega;~\left\vert F_{\tau}\left(
z\right)  \right\vert <\varepsilon\right\}  $. On choisit $\varepsilon_{0}$ de
sorte que $b\Omega^{\varepsilon_{0}}$ soit lisse. Etant donn\'{e} que par
d\'{e}finition d'un anneau l\'{e}vi-plat, les espaces affines $\mathbb{C}%
_{t}^{2}$ coupent transversalement $N$, on obtient que $d_{z}F=\frac{\partial
F}{\partial z_{1}}dz_{1}+\frac{\partial F}{\partial z_{2}}dz_{2}$ ne s'annule
pas sur $\mathbb{E}^{n,2}\cap b\Omega^{\varepsilon_{0}}$. Dans ce qui suit, on
se donne un ouvert $V$ de $\mathbb{R}^{n}$ tel que pour tout $t\in V$,
$\mathcal{Y}^{t}$ est non vide.

\begin{proposition}
\label{P/ resdbar}On suppose que $\varphi_{t}$, $t\in V$, est la restriction
\`{a} $\mathcal{Y}^{t}$ d'une $\left(  0,1\right)  $-forme $\Phi_{t}$ de
classe $C^{\infty}$ sur $\Omega^{t}$, analytique r\'{e}elle par rapport \`{a}
$t$. Alors la formule
\[
\left(  R_{0,1}^{t}\varphi_{t}\right)  \left(  z\right)  =\frac{1}{2\pi i}%
\int_{z^{\prime}\in\mathcal{Y}_{0}^{t}}K_{t}\left(  z^{\prime},z\right)
\wedge\varphi_{t}\left(  z^{\prime}\right)
\]
d\'{e}finit sur $\operatorname*{Reg}\mathcal{Y}_{0}^{t}$ une fonction de
classe $C^{\infty}$ telle que $\overline{\partial}R_{0,1}^{t}\varphi
_{t}=\varphi_{t}$ sur $\operatorname*{Reg}\mathcal{Y}_{0}^{t}$ et sur
$\mathcal{Y}_{0}^{t}$ au sens des courants. En outre $R_{0,1}^{t}\varphi
_{t}:\operatorname*{Reg}\mathcal{Y}_{0}^{t}\rightarrow\mathbb{C}$ est
analytique r\'{e}elle par rapport \`{a} $t$. Enfin, on dispose d'une formule
de type Cauchy-Pomp\'{e}iu~: si $\chi_{t}$ est la restriction \`{a}
$\mathcal{Y}^{t}$ d'une fonction de classe $C^{\infty}$ sur $\Omega^{t}$ \`{a}
support compact dans $\mathcal{Y}_{0}^{t}$, $R_{0,1}^{t}\overline{\partial
}\chi_{t}=\chi_{t}$ sur $\operatorname*{Reg}\mathcal{Y}_{0}^{t}$.
\end{proposition}

\begin{proof}
La formule ci-dessus est implicitement contenue dans~\cite{HeG-PoP1989} et
nous n'en proposons qu'une d\'{e}monstration abr\'{e}g\'{e}e contrairement
\`{a} la preuve de la d\'{e}pendance par rapport au param\`{e}tre $t$. Soient
$q\in\left\{  0,1,2\right\}  $ et $\Phi_{t}$ une $\left(  0,q\right)  $-forme
de classe $C^{\infty}$ sur $\Omega^{t}$ dont on note $\varphi_{t}$ la
restriction \`{a} $\mathcal{Y}_{t}$. Pour $\varepsilon\in\left]
0,\varepsilon_{0}\right]  $ tel que $\left\{  \left\vert F_{t}\right\vert
=\varepsilon\right\}  $ est lisse, on pose $D_{\varepsilon}^{t}=\Omega^{\ast
}\cap\Omega^{\varepsilon,t}$, $\partial_{1}D_{\varepsilon}^{t}=\overline
{\Omega^{\ast}}\cap\partial\left\{  \left\vert F_{t}\right\vert =\varepsilon
\right\}  $ et $\partial_{0}D_{\varepsilon}^{t}=\partial\Omega^{\ast}%
\cap\overline{D_{\varepsilon}^{t}}$. Pour \'{e}crire les formules
int\'{e}grales dont nous utilisons les notations suivantes~:%
\begin{gather*}
b^{\left[  0\right]  }\left(  \zeta,z\right)  =\frac{Q^{\left[  0\right]
}\left(  \zeta,z\right)  }{\left\langle Q_{0}\left(  \zeta,z\right)
,\zeta-z\right\rangle }~\&~b^{\left[  1\right]  }\left(  t,\zeta,z\right)
=\frac{Q_{t}\left(  \zeta,z\right)  }{\left\langle Q_{t}\left(  \zeta
,z\right)  ,\zeta-z\right\rangle }=\frac{Q_{t}\left(  \zeta,z\right)  }%
{F_{t}\left(  \zeta\right)  -F_{t}\left(  z\right)  }\\
\beta\left(  \zeta,z\right)  =\frac{\overline{\zeta}-\overline{z}}{\left\vert
\zeta-z\right\vert ^{2}}\\
\eta_{t}^{\left[  \nu\right]  }\left(  \lambda,\zeta,z\right)  =\left(
1-\lambda\right)  \beta\left(  \zeta,z\right)  +\lambda b_{t}^{\left[
\nu\right]  }\left(  \lambda,z\right)  ,~\eta_{t,z}^{\left[  \nu\right]
}\left(  \lambda,\zeta\right)  =\eta_{t}^{\left[  \nu\right]  }\left(
\lambda,\zeta,z\right)  ,~\nu=0,1,\\
W\left(  \zeta\right)  =%
{\textstyle\bigwedge\limits_{1\leqslant j\leqslant2}}
d\zeta_{j}~~\&~~W^{\prime}\left(  \eta\right)  =%
{\textstyle\sum\limits_{1\leqslant j\leqslant2}}
\left(  -1\right)  ^{j+1}\eta_{j}%
{\textstyle\bigwedge\limits_{k\neq j}}
d\eta_{k}\overset{d\acute{e}f}{=}%
{\textstyle\sum\limits_{1\leqslant j\leqslant2}}
\left(  -1\right)  ^{j+1}\eta_{j}d\eta_{\widehat{j}}%
\end{gather*}
et pour $z\in D_{\varepsilon}^{t}$,%
\begin{align*}
\left(  I_{D_{\varepsilon}^{t}}^{b\nu}\Phi_{t}\right)  \left(  z\right)   &
=\frac{1}{\left(  2\pi i\right)  ^{2}}\int_{\left(  \lambda,\zeta\right)
\in\left[  0,1\right]  \times\partial_{\nu}D_{\varepsilon}^{t}}\Phi_{t}\left(
\zeta\right)  \wedge W^{\prime}\left(  \eta_{t}^{\left[  \nu\right]  }\left(
\lambda,\zeta,z\right)  \right)  \wedge W\left(  \zeta\right)  ,\ \nu=0,1,\\
\left(  I_{D_{\varepsilon}^{t}}^{b\nu}\Phi_{t}\right)  \left(  z\right)   &
=\left(  I_{D_{\varepsilon}^{t}}^{b0}\Phi_{t}\right)  \left(  z\right)
+\left(  I_{D_{\varepsilon}^{t}}^{b1}\Phi_{t}\right)  \left(  z\right)  ,\\
\left(  B_{D_{\varepsilon}^{t}}\Phi_{t}\right)  \left(  z\right)   &
=\frac{1}{\left(  2\pi i\right)  ^{2}}\int_{\zeta\in D_{\varepsilon}^{t}}%
\Phi_{t}\left(  \zeta\right)  \wedge W^{\prime}\left(  \beta\left(
\zeta,z\right)  \right)  \wedge W\left(  \zeta\right)  ,\\
\left(  R_{D_{\varepsilon}^{t}}\Phi_{t}\right)  \left(  z\right)   &  =\left(
B_{D_{\varepsilon}^{t}}\Phi_{t}\right)  \left(  z\right)  +\left(
I_{D_{\varepsilon}^{t}}^{b}\Phi_{t}\right)  \left(  z\right)  .
\end{align*}
L'existence de ces int\'{e}grales ne posant pas de probl\`{e}me car le lieu
singulier de $\mathcal{Y}^{t}$ ne rencontre pas $b\Omega^{\ast}$ puisque $N$
\ \'{e}t\'{e} suppos\'{e} lisse. On d\'{e}duit de~\cite{HeG1971}
et~\cite{PoP1971} (voir aussi \cite{HeG-LeJ1984L}) que%
\begin{equation}
\forall z\in z\in D_{\varepsilon}^{t},~\left(  -1\right)  ^{q}\Phi_{t}\left(
z\right)  =\left(  L_{b}^{t,\varepsilon}\Phi_{t}\right)  \left(  z\right)
-\left(  R_{D_{\varepsilon}^{t}}\overline{\partial}\Phi_{t}\right)  \left(
z\right)  +\left(  \overline{\partial}R_{D_{\varepsilon}^{t}}\Phi_{t}\right)
\left(  z\right)  \label{HOM}%
\end{equation}
o\`{u} $L_{b}^{t,\varepsilon}\Phi$ est un terme qui vaut $0$ lorsque $q=1$ ou
que le support de $\Phi_{t}$ ne rencontre pas $b\mathcal{Y}_{0}^{t}$.\medskip

On suppose $q=1$ et on fixe $z$ dans $\mathcal{Y}_{0}^{t}\cap D_{\varepsilon
}^{t}$. La seule composante de $\Phi_{t}\wedge W^{\prime}\left(  \eta
_{t,z}^{\left[  \nu\right]  }\right)  \wedge W$ qui ne donne pas a priori une
int\'{e}grale nulle sur $\left[  0,1\right]  \times\partial_{\nu
}D_{\varepsilon}^{t}$ est la forme $\Psi_{t,z}^{\left[  \nu\right]  }$
d\'{e}finie par%
\[
\Psi_{t,z}^{\left[  \nu\right]  }=\Phi_{t}\wedge W_{\lambda}^{\prime}\left(
\eta_{t,z}^{\left[  \nu\right]  }\right)  \wedge W
\]
o\`{u} si $x$ est l'une des variables $\lambda$ ou $\zeta$, la notation
$W_{x}^{\prime}\left(  \eta_{t,z}\right)  $ signifie que dans la formule
d\'{e}finissant $W^{\prime}\left(  \eta_{t,z}\right)  $ on ne calcule de
diff\'{e}rentielle que par rapport \`{a} $x$. Un calcul direct donne que
\[
\int_{0}^{1}W_{\lambda}^{\prime}\left(  \eta_{t,z}^{\left[  \nu\right]
}\right)  =\int_{0}^{1}%
{\textstyle\sum\limits_{1\leqslant j\leqslant2}}
\left(  -1\right)  ^{j+1}\eta_{t,z,j}^{\left[  \nu\right]  }d_{\lambda}%
\eta_{t,z,\widehat{j}}^{\left[  \nu\right]  }=\det\left(  \beta,b_{t}^{\left[
\nu\right]  }\right)  \left\vert _{\left(  .,z\right)  }\right.  .
\]
Par ailleurs $W$ se factorisant sous la forme $\omega_{t}\wedge dF_{t}$ sur
$\left\{  \left\vert F_{t}\right\vert =\varepsilon\right\}  $, il vient
\begin{align*}
\left(  I_{D_{\varepsilon}^{t}}^{b1}\Phi_{t}\right)  \left(  z\right)   &
=\frac{1}{\left(  2\pi i\right)  ^{2}}\int_{\zeta\in\left\{  \left\vert
F_{t}\right\vert =\varepsilon\right\}  \cap\Omega^{\ast}}k_{t}\left(
\zeta,z\right)  \Phi_{t}\left(  \zeta\right)  \wedge\omega_{t}\left(
\zeta\right)  \wedge\frac{dF_{t}\left(  \zeta\right)  }{F_{t}\left(
\zeta\right)  }\\
\left(  I_{D_{\varepsilon}^{t}}^{b0}\Phi_{t}\right)  \left(  z\right)   &
=\frac{1}{\left(  2\pi i\right)  ^{2}}\int_{\zeta\in\left\{  \left\vert
F_{t}\right\vert \leqslant\varepsilon\right\}  \cap\partial\Omega^{\ast}}%
\det\left(  \beta\left(  \zeta,z\right)  ,b^{\left[  0\right]  }\left(
\zeta,z\right)  \right)  \Phi_{t}\left(  \zeta\right)  \wedge W\left(
\zeta\right)
\end{align*}
Par ce que $\operatorname*{Sing}\mathcal{Y}^{t}$ ne rencontre par
$b\Omega^{\ast}$, $\underset{\varepsilon\longrightarrow0^{+}}{\lim}\left(
I_{D_{\varepsilon}^{t}}^{b0}\Phi_{t}\right)  \left(  z\right)  $ existe et
vaut $0$. Par ailleurs, les r\'{e}sultats de Coleff et
Herrera~\cite{CoN-HeM1978} appliqu\'{e}s par~\cite{HeG-PoP1989} donnent tout
d'abord que $\underset{\varepsilon\longrightarrow0^{+}}{\lim}I_{D_{\varepsilon
}^{t}}^{b1}\Phi_{t}$ existe au sens des courants et vaut $-R_{0,1}^{t}%
\varphi=\frac{1}{2\pi i}\int_{\mathcal{Y}_{0}^{t}}k_{t}\left(  \zeta,.\right)
\varphi_{t}\left(  \zeta\right)  \wedge\omega_{t}\left(  \zeta\right)  $,
cette int\'{e}grale \'{e}tant \`{a} comprendre au sens des valeurs
principales, c'est \`{a} dire comme la limite lorsque $\delta$ tend vers
$0^{+}$ de $\int_{\mathcal{Y}_{0}^{t}\cap\left\{  \left\vert d_{z}%
F_{t}\right\vert \geqslant\delta\right\}  }k_{t}\left(  \zeta,.\right)
\varphi_{t}\left(  \zeta\right)  \wedge\omega_{t}\left(  \zeta\right)  $. Il
s'ensuit que $R_{0,1}^{t}\varphi_{t}=-\underset{\varepsilon\longrightarrow
0^{+}}{\lim}R_{D_{\varepsilon}^{t}}\Phi_{t}$ v\'{e}rifie $\overline{\partial
}\left(  R_{0,1}^{t}\varphi_{t}\right)  =\varphi_{t}$ sur $\operatorname*{Reg}%
\mathcal{Y}^{t}$ et sur $\mathcal{Y}^{t}$ au sens des courants. Les m\^{e}mes
r\'{e}f\'{e}rences donnent aussi que sur $\operatorname*{Reg}\mathcal{Y}^{t}$,
$R_{0,1}^{t}\varphi_{t}$ est de classe $C^{\infty}$.

Il reste \`{a} prouver l'analyticit\'{e} de $R_{0,1}^{t}\varphi_{t}$ par
rapport \`{a} $t$. Fixons $z$ dans $\operatorname*{Reg}\mathcal{Y}_{0}^{t}\cap
D_{\varepsilon}^{t}$. Un calcul \'{e}l\'{e}mentaire donne
\[
d_{\zeta}\Psi_{t,z}^{\left[  \nu\right]  }\wedge W=d_{\lambda}\left(  \Phi
_{t}\wedge W_{\zeta}^{\prime}\left(  \eta_{t,z}^{\left[  \nu\right]  }\right)
\wedge W\right)  +\overline{\partial}_{\zeta}\Phi_{t}\wedge W_{\lambda
}^{\prime}\left(  \eta_{t,z}^{\left[  \nu\right]  }\right)  \wedge W.
\]
Puisque%
\[
\left(  I_{D_{\varepsilon_{0}}^{t}}^{b1}\Phi_{t}\right)  \left(  z\right)
-\left(  I_{D_{\varepsilon}^{t}}^{b1}\Phi_{t}\right)  \left(  z\right)
=\frac{1}{\left(  2\pi i\right)  ^{2}}\int_{\left[  0,1\right]  }%
(\int_{\partial D_{\varepsilon_{0}}^{t}}\Psi_{t,z}^{\left[  1\right]  }%
-\int_{\partial D_{\varepsilon}^{t}}\Psi_{t,z}^{\left[  1\right]  }%
-\int_{\left\{  \varepsilon\leqslant\left\vert F_{t}\right\vert \leqslant
\varepsilon_{0}\right\}  \cap b\Omega^{\ast}}\Psi_{t,z}^{\left[  1\right]  })
\]
la formule de Stokes livre%
\[
\left(  I_{D_{\varepsilon_{0}}^{t}}^{b}\Phi_{t}\right)  \left(  z\right)
-\left(  I_{D_{\varepsilon}^{t}}^{b}\Phi_{t}\right)  \left(  z\right)
=\left(  A_{D_{\varepsilon_{0}}^{t}\backslash D_{\varepsilon}^{t}}\Phi
_{t}\right)  \left(  z\right)  -\left(  A_{D_{\varepsilon_{0}}^{t}\backslash
D_{\varepsilon}^{t}}^{\prime}\Phi_{t}\right)  \left(  z\right)  -\left(
J_{D_{\varepsilon_{0}}^{t}\backslash D_{\varepsilon}^{t}}^{b0}\Phi_{t}\right)
\left(  z\right)
\]
avec
\begin{align*}
\left(  A_{D_{\varepsilon_{0}}^{t}\backslash D_{\varepsilon}^{t}}\Phi
_{t}\right)  \left(  z\right)   &  =\frac{1}{\left(  2\pi i\right)  ^{2}}%
\int_{D_{\varepsilon_{0}}^{t}\backslash D_{\varepsilon}^{t}}\int_{\left[
0,1\right]  }d_{\lambda}\Psi_{t,z}^{\left[  1\right]  }\\
\left(  A_{D_{\varepsilon_{0}}^{t}\backslash D_{\varepsilon}^{t}}^{\prime}%
\Phi_{t}\right)  \left(  z\right)   &  =\frac{-1}{\left(  2\pi i\right)  ^{2}%
}\int_{\left[  0,1\right]  }\int_{D_{\varepsilon_{0}}^{t}\backslash
D_{\varepsilon}^{t}}\overline{\partial}\Phi_{t}\wedge W_{\lambda}^{\prime
}\left(  \eta_{t,z}^{\left[  1\right]  }\right)  \wedge W\\
\left(  J_{D_{\varepsilon_{0}}^{t}\backslash D_{\varepsilon}^{t}}^{b0}\Phi
_{t}\right)  \left(  z\right)   &  =\frac{1}{\left(  2\pi i\right)  ^{2}}%
\int_{\left[  0,1\right]  }\int_{\left\{  \varepsilon\leqslant\left\vert
F_{t}\right\vert \leqslant\varepsilon_{0}\right\}  \cap b\Omega^{\ast}}\left(
\Psi_{t,z}^{\left[  1\right]  }-\Psi_{t,z}^{\left[  0\right]  }\right)
\end{align*}
Etant donn\'{e} que $\eta_{t,z}^{\left[  1\right]  }\left(  1,.\right)  $ est
holomorphe, $W_{\zeta}^{\prime}\left(  \eta_{t,z}^{\left[  1\right]  }\right)
=0$. Avec $\eta_{t,z}^{\left[  1\right]  }\left(  0,.,\right)  =\beta\left(
.,z\right)  $, on obtient%
\[
\left(  A_{D_{\varepsilon_{0}}^{t}\backslash D_{\varepsilon}^{t}}\Phi
_{t}\right)  \left(  z\right)  =\frac{1}{\left(  2\pi i\right)  ^{2}}%
\int_{D_{\varepsilon_{0}}^{t}\backslash D_{\varepsilon}^{t}}\Phi_{t}\wedge
W_{\zeta}^{\prime}\left(  \beta\left(  .,z\right)  \right)  \wedge W
\]
D'o\`{u}%
\[
\left(  I_{D_{\varepsilon_{0}}^{t}}^{b}\Phi_{t}\right)  \left(  z\right)
-\left(  I_{D_{\varepsilon}^{t}}^{b}\Phi_{t}\right)  \left(  z\right)
=-\left(  B_{D_{\varepsilon_{0}}^{t}}\Phi_{t}\right)  \left(  z\right)
+\left(  B_{D_{\varepsilon}^{t}}\Phi_{t}\right)  \left(  z\right)  -\left(
J_{D_{\varepsilon_{0}}^{t}\backslash D_{\varepsilon}^{t}}^{b0}\Phi_{t}\right)
\left(  z\right)
\]
avec
\begin{align*}
\left(  J_{D_{\varepsilon_{0}}^{t}\backslash D_{\varepsilon}^{t}}^{b0}\Phi
_{t}\right)  \left(  z\right)   &  =\frac{1}{\left(  2\pi i\right)  ^{2}}%
\int_{\left\{  \varepsilon\leqslant\left\vert F_{t}\right\vert \leqslant
\varepsilon_{0}\right\}  \cap b\Omega^{\ast}}\det\left[  \beta\left(
.,z\right)  ,b_{t}^{\left[  1\right]  }\left(  .,z\right)  -b^{\left[
0\right]  }\left(  .,z\right)  \right]  \Phi_{t}\wedge W\\
\left(  A_{D_{\varepsilon_{0}}^{t}\backslash D_{\varepsilon}^{t}}^{\prime}%
\Phi_{t}\right)  \left(  z\right)   &  =\frac{-1}{\left(  2\pi i\right)  ^{2}%
}\int_{D_{\varepsilon_{0}}^{t}\backslash D_{\varepsilon}^{t}}\det\left(
\beta\left(  .,z\right)  ,b_{t}^{\left[  1\right]  }\left(  .,z\right)
\right)  \overline{\partial}\Phi_{t}\wedge W
\end{align*}

Etant donn\'{e} que $\left\{  F_{t}=0\right\}  $ est lisse pr\`{e}s de
$b\Omega^{\ast}$, on peut choisir $F_{t}$ comme premi\`{e}re coordonn\'{e}e
d'un syst\`{e}me de coordonn\'{e}es de $\mathbb{C}^{2}$ qui d\'{e}pend
holomorphiquement de la complexification naturelle de la variable r\'{e}elle
$t$ en une variable complexe $\tau$ de $\mathbb{C}^{n}$. On constate alors
qu'avec des notations \'{e}videntes, $\left(  J_{D_{\varepsilon}^{\tau}}%
^{b0}\Phi\right)  \left(  z\right)  $ est bien d\'{e}finie pour tout
$\varepsilon\in\left]  0,\varepsilon_{0}\right]  $ et $\tau$ suffisamment
voisin de $t$, que $\left\vert \left(  J_{D_{\varepsilon}^{\tau}}^{b}%
\Phi_{\tau}\right)  \left(  z^{\prime}\right)  \right\vert $ est un $O\left(
\varepsilon\right)  $ uniform\'{e}ment par rapport \`{a} $\left(
\tau,z^{\prime}\right)  $ quand $\left(  \tau,z^{\prime}\right)  $ varie dans
un voisinage suffisamment petit de $\left(  t,z\right)  $ dans $\mathbb{C}%
^{n}\times\operatorname*{Reg}\mathcal{Y}_{0}^{t}$ et que
\[
\left(  J_{D_{\varepsilon_{0}}^{t}\backslash D_{\varepsilon}^{t}}^{b0}\Phi
_{t}\right)  \left(  z\right)  =\left(  J_{D_{\varepsilon_{0}}^{t}}^{b0}%
\Phi_{t}\right)  \left(  z\right)  -\left(  J_{D_{\varepsilon_{0}}%
^{t}\backslash D_{\varepsilon}^{t}}^{b0}\Phi_{t}\right)  \left(  z\right)  .
\]

De m\^{e}me, gr\^{a}ce aux r\'{e}sultats de~\cite{CoN-HeM1978}, on peut
d\'{e}finir $\left(  A_{D_{\varepsilon}^{t}}^{\prime}\Phi_{t}\right)  \left(
z\right)  $ pour $\varepsilon\in\left]  0,\varepsilon_{0}\right]  $ comme la
limite de $\left(  A_{D_{\varepsilon}^{t}\backslash D_{\eta}^{t}}^{\prime}%
\Phi_{t}\right)  \left(  z\right)  $ lorsque $\eta\downarrow0^{+}$ et obtenir
aussi que $\left(  A_{D_{\varepsilon}^{t}}^{\prime}\Phi_{\tau}\right)  \left(
z^{\prime}\right)  =_{\varepsilon\downarrow0^{+}}o\left(  1\right)  $
uniform\'{e}ment en $\left(  \tau,z^{\prime}\right)  $ lorsque $\left(
\tau,z^{\prime}\right)  $ varie dans un voisinage suffisamment petit de
$\left(  t,z\right)  $ dans $\mathbb{C}^{n}\times\operatorname*{Reg}%
\mathcal{Y}$. Ecrivant
\[
\left(  A_{D_{\varepsilon_{0}}^{t}\backslash D_{\varepsilon}^{t}}^{\prime}%
\Phi_{t}\right)  \left(  z\right)  =\left(  A_{D_{\varepsilon_{0}}^{t}%
}^{\prime}\Phi_{t}\right)  \left(  z\right)  -\left(  A_{D_{\varepsilon}^{t}%
}^{\prime}\Phi_{t}\right)  \left(  z\right)
\]
on obtient qu'en fin de compte,%
\begin{equation}
\left(  R_{D_{\varepsilon_{0}}^{t}}\Phi_{t}\right)  \left(  z\right)  +\left(
J_{D_{\varepsilon_{0}}^{t}}^{b0}\Phi_{t}\right)  \left(  z\right)  +\left(
A_{D_{\varepsilon_{0}}^{t}}^{\prime}\Phi_{t}\right)  \left(  z\right)
=\left(  R_{D_{\varepsilon}^{t}}\Phi_{t}\right)  \left(  z\right)  +\left(
A_{D_{\varepsilon}^{t}}^{\prime}\Phi_{t}\right)  \left(  z\right)  \left(
J_{D_{\varepsilon}^{t}}^{b0}\Phi_{t}\right)  \left(  z\right)  \label{Inv}%
\end{equation}
En passant \`{a} la limite lorsque $\varepsilon$ tend vers $0^{+}$, il vient%
\begin{equation}
\left(  R_{D_{\varepsilon_{0}}^{t}}\Phi_{t}\right)  \left(  z\right)  +\left(
J_{D_{\varepsilon_{0}}^{t}}^{b0}\Phi_{t}\right)  \left(  z\right)  +\left(
A_{D_{\varepsilon_{0}}^{t}}^{\prime}\Phi_{t}\right)  \left(  z\right)
=\left(  R_{0,1}^{t}\varphi\right)  \left(  z\right)  \text{.} \label{Inv2}%
\end{equation}

Etant donn\'{e} que $\left\{  \left\vert F_{t}\right\vert =\varepsilon
_{0}\right\}  $ est lisse, $\left(  R_{D_{\varepsilon_{0}}^{\tau}}\Phi_{\tau
}\right)  \left(  z^{\prime}\right)  $ d\'{e}pend analytiquement de $\left(
\tau,z^{\prime}\right)  $ au voisinage de $\left(  t,z\right)  $ dans
$\mathbb{C}^{n}\times\operatorname*{Reg}\mathcal{Y}$. Puisque $\left(
J_{D_{\varepsilon_{0}}^{\tau}}^{b0}\Phi_{\tau}\right)  \left(  z^{\prime
}\right)  +\left(  A_{D_{\varepsilon_{0}}^{\tau}}^{\prime}\Phi_{\tau}\right)
\left(  z^{\prime}\right)  =_{\varepsilon_{0}\downarrow0^{+}}o\left(
1\right)  $ uniform\'{e}ment en $\left(  \tau,z^{\prime}\right)  $ au
voisinage de $\left(  t,z\right)  $ dans $\mathbb{C}^{n}\times
\operatorname*{Reg}\mathcal{Y}$, il appara\^{\i}t gr\^{a}ce au
th\'{e}or\`{e}me de Weierstrass sur les limites de fonctions holomorphes que
$\left(  R_{0,1}^{\tau}\varphi_{\tau}\right)  \left(  \tau,z^{\prime}\right)
$ est analytique au voisinage de $\left(  t,z\right)  $ dans $\mathbb{R}%
^{n}\times\operatorname*{Reg}\mathcal{Y}$.\medskip

Lorsque $q=0$ et que $b\mathcal{Y}_{0}^{t}\cap\operatorname*{Supp}\Phi
_{t}=\varnothing$, la formule (\ref{HOM}) devient $\Phi_{t}=R_{D_{\varepsilon
}^{t}}\overline{\partial}\Phi_{t}$ et ce qui pr\'{e}c\`{e}de montre qu'on peut
passer \`{a} la limite pour obtenir $\Phi_{t}\left\vert _{\operatorname*{Reg}%
\mathcal{Y}^{t}}\right.  =R_{D_{\varepsilon}^{t}}\overline{\partial}\Phi
_{t}\left\vert _{\operatorname*{Reg}\mathcal{Y}^{t}}\right.  $ sur
$\operatorname*{Reg}\mathcal{Y}^{t}$.\medskip
\end{proof}

Le lemme ci-dessous est un compl\'{e}ment n\'{e}cessaire pour \'{e}tablir
l'existence de fonctions de Green.

\begin{lemma}
\label{L/ resdbar}Les notations \'{e}tant celles de la
proposition~\ref{P/ resdbar}, on suppose qu'il existe un voisinage ouvert
$G^{\ast}$ de $M$ dans $\mathbb{E}^{n,2}$ tel que pour tout $t\in V$,
$\operatorname*{supp}\varphi_{t}\subset G^{\ast,t}$ et $G_{0}^{\ast,t}%
\cap\left(  \operatorname*{Sing}\mathcal{Y}_{0}^{t}\cup b\mathcal{Y}_{0}%
^{t}\right)  =\varnothing$. $R_{0,1}^{t}\varphi_{t}$ est alors la restriction
\`{a} $\mathcal{Y}_{0}^{t}$ d'une fonction lisse de $\Omega^{t}$ qui
d\'{e}pend analytiquement de $t$.
\end{lemma}

\begin{proof}
On peut se ramener au cas o\`{u} pour tout $t\in V$ le support de $\Phi_{t}$
est contenu aussi dans $G^{\ast}\cap\Omega^{t}$. Le terme $\left(
J_{D_{\varepsilon_{0}}^{t}}^{b0}\Phi_{t}\right)  \left(  z\right)  $
dans~(\ref{Inv2}) est alors nul tandis que%
\[
\left(  A_{D_{\varepsilon_{0}}^{t}}^{\prime}\Phi_{t}\right)  \left(  z\right)
=\frac{-1}{\left(  2\pi i\right)  ^{2}}\int_{G^{\ast}\cap D_{\varepsilon_{0}%
}^{t}}\det\left(  \beta\left(  .,z\right)  ,Q_{t}\left(  .,z\right)  \right)
\frac{\overline{\partial}\Phi_{t}}{F_{t}}\wedge W
\]
Etant donn\'{e} que $G^{\ast}\cap D_{\varepsilon_{0}}^{t}\subset
\operatorname*{Reg}\mathcal{Y}^{t}$, il est standard que dans $G^{\ast}$ on
peut factoriser $\overline{\partial}\Phi_{t}$ sous la forme $F_{t}\Xi_{t}$
o\`{u} $\Xi_{t}$ est une $\left(  0,1\right)  $-forme lisse de $G^{\ast}$ qui
d\'{e}pend analytiquement de $t$ et \'{e}crire%
\[
\left(  A_{D_{\varepsilon_{0}}^{t}}^{\prime}\Phi_{t}\right)  \left(  z\right)
=\frac{-1}{\left(  2\pi i\right)  ^{2}}\int_{G^{\ast}\cap D_{\varepsilon_{0}%
}^{t}}\det\left(  \beta\left(  .,z\right)  ,Q_{t}\left(  .,z\right)  \right)
\Xi_{t}\wedge W
\]
Sous cette forme, il appara\^{\i}t que $\left(  A_{D_{\varepsilon_{0}}^{t}%
}^{\prime}\Phi_{t}\right)  \left(  z\right)  $ est en fait la valeur en $z$ de
la restriction \`{a} $\mathcal{Y}_{0}^{t}$ d'une fonction lisse de $\Omega
^{t}$ qui d\'{e}pend analytiquement de $t$. Pour des raisons de support,
\[
\left(  R_{D_{\varepsilon_{0}}^{t}}\Phi_{t}\right)  \left(  z\right)  =\left(
B_{D_{\varepsilon_{0}}^{t}}\Phi_{t}\right)  \left(  z\right)  +\left(
I_{D_{\varepsilon_{0}}^{t}}^{b1}\Phi_{t}\right)  \left(  z\right)
\]
et il appara\^{\i}t que cette conclusion est valide pour $\left(
R_{D_{\varepsilon_{0}}^{t}}\Phi_{t}\right)  \left(  z\right)  $ et donc in
fine, avec (\ref{Inv2}), pour $\left(  R_{0,1}^{t}\varphi_{t}\right)  \left(
z\right)  $.\medskip
\end{proof}

Pour \'{e}noncer la proposition nous devons pr\'{e}ciser quelques notations.
Si $B$ est une sous-vari\'{e}t\'{e} r\'{e}elle compacte de $\mathbb{C}^{2}$ et
si $U$ est un ouvert d'une courbe complexe $Y$ de $\mathbb{C}^{2}\backslash B$
bord\'{e}e au sens des courants de $\mathbb{C}^{2}$ par $B$, on note
$\mathcal{E}_{p,q}\left(  U\right)  $ l'espace des restrictions \`{a} $U$ de
$\left(  p,q\right)  $-formes de $\mathbb{C}^{2}$ de classe $C^{\infty}$~; on
note $\mathcal{D}_{p,q}\left(  U\right)  $ l'espace de ces formes dont le
support est contenu dans $U$. On note $\widetilde{\mathcal{E}}$ $_{p,q}\left(
U\right)  $ l'espace des $\left(  p,q\right)  $-formes de classe $C^{\infty}$
sur $\operatorname*{Reg}Y$ et qui au voisinage de $\operatorname*{Sing}Y$ sont
des courants, c'est \`{a} dire qui sont dans le dual de $\mathcal{E}%
_{p,q}\left(  V\right)  $ pour un certain voisinage $V$ de
$\operatorname*{Sing}Y$ dans $Y$.

\begin{proposition}
\label{P/ fction de Green}Si $z_{\ast}\in\operatorname*{Reg}\mathcal{Y}%
_{0}^{t}$, l'expression%
\begin{equation}
g_{t,z_{\ast}}\left(  z\right)  =\frac{1}{4\pi^{2}}\int_{z^{\prime}%
\in\mathcal{Y}_{0}^{t}}\overline{k_{t}\left(  z^{\prime},z\right)  }%
k_{t}\left(  z_{\ast},z^{\prime}\right)  \omega_{t}\left(  z^{\prime}\right)
\wedge\overline{\omega_{t}}\left(  z^{\prime}\right)  . \label{F/ Green}%
\end{equation}
est bien d\'{e}finie au sens des valeurs principales quand $z\in\left(
\operatorname*{Reg}\mathcal{Y}_{0}^{t}\right)  \backslash\left\{  z_{\ast
}\right\}  $. La fonction $g_{t,z_{\ast}}$ se prolonge \`{a} $\mathcal{Y}%
_{0}^{t}$ comme un courant et v\'{e}rifie alors $i\partial\overline{\partial
}g_{t,z_{\ast}}=\Delta_{z_{\ast}}$ o\`{u} $\Delta_{z_{\ast}}=\delta_{z_{\ast}%
}dV$, $dV=i\partial\overline{\partial}\left\vert .\right\vert ^{2}$ et
$\delta_{z_{\ast}}$ est la mesure de Dirac port\'{e}e par $\left\{  z_{\ast
}\right\}  $. En outre, si $G^{\ast}$ un voisinage ouvert de $M$ dans
$\mathbb{E}^{n,2}$ tel que pour tout $t\in V$, $G^{\ast,t}\cap\left(
\operatorname*{Sing}\mathcal{Y}_{0}^{t}\cup b\mathcal{Y}_{0}^{t}\right)
=\varnothing$, $g_{t,z_{\ast}}\left\vert _{G^{\ast}\cap\mathcal{Y}_{0}^{t}%
}\right.  $ est un courant qui d\'{e}pend analytiquement de $t$.
\end{proposition}

\begin{proof}
D'apr\`{e}s la proposition pr\'{e}c\'{e}dente, on dispose d'un op\'{e}rateur
$R_{0,1}^{t}:\mathcal{E}_{0,1}\left(  \mathcal{Y}_{0}^{t}\right)
\longrightarrow\widetilde{\mathcal{E}}_{0,0}\left(  \mathcal{Y}_{0}%
^{t}\right)  $ qui v\'{e}rifie $\overline{\partial}R_{0,1}^{t}\varphi=\varphi$
pour tout $\varphi\in\mathcal{E}_{0,1}\left(  \mathcal{Y}_{0}^{t}\right)  $.
Puisque $\Delta_{z_{\ast}}$ est lisse (et m\^{e}me nul) au voisinage de
$\operatorname*{Sing}\mathcal{Y}_{0}^{t}$, $\Delta_{z_{\ast}}$ est un courant
qui agit sur $\widetilde{\mathcal{E}}_{0,0}\left(  \mathcal{Y}_{0}^{t}\right)
$ et $\left(  R_{0,1}^{t}\right)  ^{\ast}:\widetilde{\mathcal{E}}_{0,0}\left(
\mathcal{Y}_{0}^{t}\right)  ^{\prime}\longrightarrow\mathcal{E}_{0,1}\left(
\mathcal{Y}_{0}^{t}\right)  ^{\prime}$ agit sur $\Delta_{z_{\ast}}$. $\left(
R_{0,1}^{t}\right)  ^{\ast}\Delta_{z_{\ast}}$ v\'{e}rifie $\overline{\partial
}\left(  R_{0,1}^{t}\right)  ^{\ast}\Delta_{z_{\ast}}=\Delta_{z_{\ast}}$ au
sens des courants car si $\chi\in\mathcal{D}_{0,0}\left(  \mathcal{Y}_{0}%
^{t}\right)  $, $R_{0,1}^{t}\overline{\partial}\chi=\chi$ et donc%
\[
\left\langle \overline{\partial}\left(  R_{0,1}^{t}\right)  ^{\ast}%
\Delta_{z_{\ast}},\chi\right\rangle =\left\langle \Delta_{z_{\ast}}%
,R_{0,1}^{t}\overline{\partial}\chi\right\rangle =\left(  R_{0,1}^{t}%
\overline{\partial}\chi\right)  \left(  z_{\ast}\right)  =\chi\left(  z_{\ast
}\right)  .
\]
Si $\varphi\in\mathcal{E}_{0,1}\left(  \mathcal{Y}_{0}^{t}\right)  $%
\begin{align*}
\left\langle \left(  R_{0,1}^{t}\right)  ^{\ast}\Delta_{z_{\ast}}%
,\varphi\right\rangle  &  =\left\langle \Delta_{z_{\ast}},R_{0,1}^{t}%
\varphi\right\rangle =\left(  R_{0,1}^{t}\varphi\right)  \left(  z_{\ast
}\right) \\
&  =\int_{z^{\prime}\in\mathcal{Y}_{0}^{t}}k_{t}\left(  z^{\prime},z_{\ast
}\right)  \omega_{t}\left(  z^{\prime}\right)  \wedge\varphi\left(  z^{\prime
}\right)  =\left\langle k_{t}\left(  .,z_{\ast}\right)  \omega_{t}%
,\varphi\right\rangle
\end{align*}
D'o\`{u} $\left(  R_{0,1}^{t}\right)  ^{\ast}\Delta_{z_{\ast}}=k_{t,z_{\ast}%
}\omega_{t}$ o\`{u} $k_{t,z_{\ast}}=k_{t}\left(  .,z_{\ast}\right)  $.

Soit $G^{\ast}$ comme dans l'\'{e}nonc\'{e}. Compte tenu du
lemme~\ref{L/ resdbar}, $R_{0,1}^{t}$ applique $\mathcal{D}_{0,1}\left(
G^{\ast}\cap\mathcal{Y}_{0}^{t}\right)  $ dans $\mathcal{E}_{0,0}\left(
\mathcal{Y}_{0}^{t}\right)  $ de sorte que pour $\Theta_{t}=k_{t,z_{\ast}%
}\omega_{t}\wedge\overline{\omega_{t}}$ qui est une $\left(  1,1\right)
$-forme dont le support singulier est $\left\{  z_{\ast}\right\}
\cup\operatorname*{Sing}\mathcal{Y}_{0}^{t}$, on peut d\'{e}finir sur
$\mathcal{Y}_{0}^{t}$ un \'{e}l\'{e}ment $g_{t,z_{\ast}}$ de $\mathcal{D}%
_{0,0}\left(  G^{\ast}\cap\mathcal{Y}_{0}^{t}\right)  ^{\prime}$ par la
formule%
\[
\overline{g_{t,z_{\ast}}}=\omega_{t}\lrcorner\left(  R_{0,1}^{t}\right)
^{\ast}\left(  \overline{\Theta_{t}}\right)
\]
Au sens des courants sur $G^{\ast}\cap\mathcal{Y}_{0}^{t}$, on a donc
$\overline{\partial}\overline{g_{t,z_{\ast}}}\wedge\omega_{t}=\overline
{\partial}\left(  \overline{g_{t,z_{\ast}}}\omega_{t}\right)  =\overline
{\partial}\left(  R_{0,1}^{t}\right)  ^{\ast}\left(  \overline{\Theta_{t}%
}\right)  =\overline{\Theta_{t}}=\overline{k_{t,z_{\ast}}\omega_{t}}%
\wedge\omega_{t}$. Comme toute $\left(  0,1\right)  $-forme se factorise sur
$G^{\ast}\cap\mathcal{Y}_{0}^{t}$ par $\overline{\omega_{t}}$, on en
d\'{e}duit que $\partial g_{t,z_{\ast}}=k_{t,z_{\ast}}\omega_{t}$ et donc que
$i\partial\overline{\partial}g_{t,z_{\ast}}=\Delta_{z_{\ast}}\left\vert
_{G^{\ast}\cap\mathcal{Y}_{0}^{t}}\right.  $ . Puisque $\overline{\partial}$
est elliptique sur $\operatorname*{Reg}\mathcal{Y}^{t}$, $g_{t,z_{\ast}}$ est
en fait une fonction r\'{e}elle analytique sur $\left(  G^{\ast}%
\cap\mathcal{Y}_{0}^{t}\right)  \backslash\left\{  z_{\ast}\right\}  $. Si
$\chi_{t}$ est la restriction \`{a} $\mathcal{Y}_{0}^{t}$ d'une $\left(
1,1\right)  $-forme lisse de $\Omega^{t}$ dont le support est contenu dans
$G^{\ast}$, il vient
\begin{gather*}
\overline{\left\langle g_{t,z_{\ast}},\chi_{t}\right\rangle }=\left\langle
\overline{g_{t,z_{\ast}}},\omega_{t}\wedge\left(  \omega_{t}\lrcorner
\overline{\chi_{t}}\right)  \right\rangle =\left\langle \left(  R_{0,1}%
^{t}\right)  ^{\ast}\left(  \overline{\Theta_{t}}\right)  ,\omega_{t}%
\lrcorner\overline{\chi_{t}}\right\rangle =\left\langle \overline{\Theta_{t}%
},R_{0,1}^{t}\left(  \omega_{t}\lrcorner\overline{\chi_{t}}\right)
\right\rangle \\
\left\langle g_{t,z_{\ast}},\chi_{t}\right\rangle =\left\langle k_{t,z_{\ast}%
}\omega_{t},\overline{R_{0,1}^{t}\left(  \omega_{t}\lrcorner\overline{\chi
_{t}}\right)  \omega_{t}}\right\rangle
\end{gather*}
Cette derni\`{e}re formule prouve aussi que si $\mathbb{\chi}_{t}$ d\'{e}pend
analytiquement de $t$, il en est de m\^{e}me de $\left\langle g_{t,z_{\ast}%
},\chi_{t}\right\rangle $. Par ailleurs, comme%
\[
\left(  R_{0,1}^{t}\left(  \omega_{t}\lrcorner\overline{\chi_{t}}\right)
\right)  \left(  z^{\prime}\right)  =\frac{1}{\left(  2\pi i\right)  ^{2}}%
\int_{\mathcal{Y}_{0}^{t}}k_{t}\left(  .,z^{\prime}\right)  \omega_{t}%
\wedge\left(  \omega_{t}\lrcorner\overline{\chi_{t}}\right)  =\frac{1}{\left(
2\pi i\right)  ^{2}}\int_{\mathcal{Y}_{0}^{t}}k_{t}\left(  .,z^{\prime
}\right)  \overline{\chi_{t}}%
\]
on obtient%
\begin{equation}
\overline{\left\langle g_{t,z_{\ast}},\chi_{t}\right\rangle }=\frac{1}{\left(
2\pi i\right)  ^{2}}\int_{\mathcal{Y}_{0}^{t}}\left\langle \overline
{k_{t,z_{\ast}}\left(  z^{\prime}\right)  }\overline{\omega_{t}\left(
z^{\prime}\right)  },k_{t}\left(  .,z^{\prime}\right)  \omega_{t}\left(
z^{\prime}\right)  \right\rangle \overline{\chi_{t}} \label{F/ GreenChi}%
\end{equation}
ce qui implique que lorsque $z\in G^{\ast}\cap\mathcal{Y}_{0}^{t}$,%
\[
\overline{g_{t,z_{\ast}}\left(  z\right)  }=\frac{1}{\left(  2\pi i\right)
^{2}}\left\langle \overline{k_{t,z_{\ast}}}\overline{\omega_{t}},k_{t,z}%
\omega_{t}\right\rangle =\frac{1}{4\pi^{2}}\int_{z^{\prime}\in\mathcal{Y}%
_{0}^{t}}\overline{k_{t}\left(  z_{\ast},z^{\prime}\right)  }k_{t}\left(
z^{\prime},z\right)  \overline{\omega_{t}}\left(  z^{\prime}\right)
\wedge\omega_{t}\left(  z^{\prime}\right)  .
\]
On constate en particulier que $g_{t,z_{\ast}}$ ne d\'{e}pend pas de $G^{\ast
}$ et qu'on obtient ainsi une fonction d\'{e}finie et harmonique sur $\left(
\operatorname*{Reg}\mathcal{Y}_{0}^{t}\right)  \backslash\left\{  z_{\ast
}\right\}  $ v\'{e}rifiant $i\partial\overline{\partial}g_{t,z_{\ast}}%
=\Delta_{z_{\ast}}\left\vert _{\operatorname*{Reg}\mathcal{Y}_{0}^{t}}\right.
$. Ce qui ach\`{e}ve la preuve de la proposition.\medskip
\end{proof}

Revenons \`{a} la fonction $u$ consid\'{e}r\'{e}e au d\'{e}but de la preuve du
th\'{e}or\`{e}me~\ref{T/ harmo}. On se donne $N^{o}$, $N^{+}$ et $N^{-}$ comme
d\'{e}finis au d\'{e}but de la preuve du th\'{e}or\`{e}me\textit{~}%
\ref{T/bord CxR A}\textit{.} Pour $t\in V$ et $z\in\left(  \mathcal{X}\cup
N^{o}\backslash M\right)  ^{t}$, on pose
\[
U\left(  t,z\right)  =\frac{2}{i}\int_{\zeta\in M_{t}}u\left(  t,\zeta\right)
\partial_{\zeta}g_{t}\left(  z,\zeta\right)  +g_{t}\left(  z,\zeta\right)
\overline{\partial u}\left(  t,\zeta\right)
\]
o\`{u} $\partial u$ est d\'{e}finie comme plus haut et les int\'{e}grales
\'{e}tant \`{a} prendre au sens des courants. Par construction, $U$ est une
fonction r\'{e}elle analytique univalu\'{e}e et $U\left(  t,.\right)  $ est
harmonique sur $N^{o}\cup\operatorname*{Reg}\mathcal{X}^{t}$. On convient de
poser $U^{+}\left(  t,z\right)  =U\left(  t,z\right)  $ si $z\in
\mathcal{X}^{t}$ et $U^{-}\left(  t,z\right)  =U\left(  t,z\right)  $ lorsque
$z\in\left(  N^{-}\right)  ^{t}$. Bien que $\mathcal{X}_{t}$ puisse \^{e}tre
singuli\`{e}re, \cite[lemme 15]{HeG-MiV2007} s'applique et donne que les
fonctions $U^{+}\left(  t,.\right)  $ et $U^{-}\left(  t,.\right)  $ se
prolongent contin\^{u}ment \`{a} $M^{t}$ et que si $z\in M^{t}$,
\[
u\left(  t,z\right)  =U^{+}\left(  t,z\right)  -U^{-}\left(  t,z\right)  .
\]

On note $T_{1}$ l'ensemble des r\'{e}els $\tau$ tels que pour tout
param\`{e}tre $t=\left(  t_{1},...,t_{n}\right)  \in V$ tels que $t_{1}<\tau$,
$U^{-}\left(  t,z\right)  =0$ lorsque $z\in\left(  N^{-}\right)  ^{t}$.
$T_{1}\neq\varnothing$ car $M_{t}=\varnothing$ si $t\in V\backslash
\pi_{\mathbb{R}}\left(  M\right)  $. $T_{1}$ est donc un intervalle non vide
dont on note $\tau_{\ast}$ la borne sup\'{e}rieure. Supposons $\tau_{\ast
}<\tau_{\infty}=\sup\left\{  t_{1};~t\in\pi_{\mathbb{R}}\left(  V\right)
\right\}  $. Il existe alors n\'{e}cessairement un param\`{e}tre dont la
premi\`{e}re coordonn\'{e}es est $\tau_{\ast}$ et dont la section
correspondante de $M$ est non vide. Soit $t_{\ast}$ l'un de ces
param\`{e}tres. On sait que $U^{-}=0$ dans $N^{-}$ au dessus de $\left\{
t_{1}<\tau_{\ast}\right\}  $. Soit $\left(  t_{\ast},z_{\ast}\right)  $ un
point de $\overline{N_{t_{\ast}}^{-}}$. Puisque $N$ est coup\'{e}
transversalement par $\mathbb{C}_{t_{\ast}}^{2}$, un voisinage de $\left(
t_{\ast},z_{\ast}\right)  $ dans $N$ rencontre forc\'{e}ment $\left\{  \left(
t,z\right)  \in N^{-};~t_{1}<\tau_{\ast}\right\}  $ sur un ouvert non vide de
$N$. On en d\'{e}duit que $U^{-}$ est aussi nulle au voisinage de $\left(
t_{\ast},z_{\ast}\right)  $ dans $N^{-}$. Recouvrant $\overline{N_{t_{\ast}%
}^{-}}$ par un nombre fini de tels voisinages, on en d\'{e}duit que $U^{-}=0$
sur $N_{t}^{-}$ pour tout $t$ suffisamment voisin de $t_{\ast}$. Du coup,
$\tau_{\ast}<\sup T_{1}$, ce qui est une contradiction. Ainsi, $\tau_{\ast
}\geqslant\tau_{\infty}$ et $u_{t}$ se prolonge harmoniquement \`{a}
$\mathcal{X}_{t}$ pour toute valeur de $t$ par la fonction $U^{+}\left(
t,.\right)  $

\subsection{\textit{Preuve du th\'{e}or\`{e}me~\ref{T/bord CxR B}}}

Cette preuve est essentiellement la m\^{e}me que celle du
th\'{e}or\`{e}me~\ref{T/bord CxR A} et nous nous bornons \`{a} indiquer les
changements entre les deux d\'{e}monstrations. On se donne $T$, $M$ et $N$
comme dans l'\'{e}nonc\'{e} de th\'{e}or\`{e}me~\ref{T/bord CxR B}. Par
hypoth\`{e}se, $M$ contient une partie ferm\'{e}e $S$ telle que $\mathcal{H}%
^{n+1}\left(  S\right)  =0$ et $M\backslash S$ est une vari\'{e}t\'{e}
orient\'{e}e dont le courant d'int\'{e}gration est $T$.

Puisque $M$ est r\'{e}el analytique, on sait d'apr\`{e}s Hardt\cite[th.
4.3]{HaR1972}, que $\tau\mapsto T_{\tau}$ est d\'{e}finie et continu de
$Y=\left\{  \tau\in\mathbb{R}^{n}~;~\dim_{\mathbb{R}}\left(  M\cap
\mathbb{C}_{\tau}^{2}\right)  \leqslant1\right\}  $ dans l'espace des
cha\^{\i}nes localement int\'{e}grales de $N$. Pour tout param\`{e}tre $\tau$,
$\dim_{\mathbb{R}}\left(  M\cap\mathbb{C}_{\tau}^{2}\right)  \in\left\{
0,1,2\right\}  $ puisque $M\cap\mathbb{C}_{\tau}^{2}\subset N_{\tau}$.
Supposons que $M\cap\mathbb{C}_{t_{\ast}}^{2}$ est de dimension $2$ sur l'une
de ses composantes connexes $%
C%
$. Cette intersection est alors non transverse en chacun de ses points car
sinon $M$ serait au voisinage de l'un de ses point de dimension $n+2$. Soit
$p_{\ast}$ un point de $%
C%
$. Par hypoth\`{e}se, $M$ est incluse au voisinage de $p_{\ast}$ dans l'une
branche lisse de $N$~; on suppose pour ne pas introduire de notation
suppl\'{e}mentaire que $N$ est une sous-vari\'{e}t\'{e} de $\mathbb{E}^{n,2}$.
Modulo un changement de coordonn\'{e}es, on se ram\`{e}ne au cas o\`{u}
$p_{\ast}$ est l'origine~; $N$ est alors donn\'{e}e au voisinage de $p_{\ast
}=0$ par une \'{e}quation de la forme $z_{2}=f\left(  z_{1},t\right)  $ et $M$
est caract\'{e}ris\'{e}e dans ce voisinage par une \'{e}quation additionnelle
$\rho\left(  z_{1},t\right)  =0$, $f$ et $\rho$ \'{e}tant des fonctions de
classe $C^{1}$. Puisque $M$ et $\mathbb{C}_{0}^{2}$ se coupent non
transversalement au voisinage de $0$, il existe des fonctions continues
$\rho_{1},...,\rho_{n}$ telles que $\rho=\Sigma t_{j}\rho_{j}$. On en
d\'{e}duit que $%
C%
$ est ouverte dans $N_{t}$ et donc, puisque $\mathcal{C}_{t}$ est aussi
ferm\'{e}e dans $N_{t}$, que $%
C%
$ est une composante connexe de $N_{t}$. Puisque $N_{t}$ est une
sous-vari\'{e}t\'{e} de $\mathbb{C}^{2}$, on en d\'{e}duit que $b%
C%
\subset bN_{t}\subset bN$, ce qui est absurde. Ainsi, $\tau\mapsto T_{\tau}$
est d\'{e}finie et continu sur $\mathbb{R}^{n}$. En particulier,%
\[
I_{T,h}:t\mapsto\left\langle T_{t},h\right\rangle ,
\]
$h$ forme diff\'{e}rentielle fix\'{e}e du type $h_{1}dz_{1}+h_{2}dz_{2}$ avec
$h_{1},h_{2}\in\mathcal{O}\left(  \mathbb{C}^{2}\times\mathbb{C}^{n}\right)
$, est continue sur $\mathbb{R}^{n}$.

Comme dans la preuve du th\'{e}or\`{e}me~\ref{T/bord CxR A}, il s'agit dans un
premier temps de prouver l'analyticit\'{e} de $I_{T,h}$. Celle-ci \'{e}tant
triviale sur $\mathbb{R}^{n}\backslash\pi_{\mathbb{R}}\left(  M\right)  $, on
raisonne au voisinage d'un param\`{e}tre $t_{\ast}$ fix\'{e} de $M$. Notons
$\Delta_{d}$ le simplexe de dimension $d$ naturellement orient\'{e}, appelons
\textbf{cellule} analytique de dimension de $d$ un courant de la forme
$\varphi_{\ast}\left[  \Delta_{d}\right]  $ o\`{u} $\varphi$ est un
diff\'{e}omorphisme r\'{e}el analytique au voisinage de $\Delta_{d}$ et \`{a}
valeurs dans $N$. Enfin, appelons \textbf{cha\^{\i}ne cellulaire} analytique
de dimension $d$ de $N$ une somme finie de cellules de dimension $d$ de $N$.

Fixons $\varepsilon$ dans $\mathbb{R}_{+}^{\ast}$. Il r\'{e}sulte de \cite[th.
4.2.9]{FeH1969Li} et du fait que $N$ est localement une r\'{e}union de
sous-vari\'{e}t\'{e}s de $\mathbb{E}^{n,2}$ qu'on peut trouver dans $N$ une
cha\^{\i}ne cellulaire analytique $%
P%
_{\varepsilon}$ de dimension $n+1$ et un courant int\'{e}gral $\Theta
_{\varepsilon}$ de dimension $n+2$ support\'{e} par $N$ tels que
$\operatorname*{Spt}%
P%
_{\varepsilon}\cup\operatorname*{Spt}\Theta_{\varepsilon}\subset\left\{
dist\left(  ,M\right)  \leqslant\varepsilon\right\}  $ et $T=%
P%
_{\varepsilon}+d\Theta_{\varepsilon}$. Si le support d'une cellule
$\varphi_{\ast}\left[  \Delta_{d}\right]  $ intervenant dans $%
P%
_{\varepsilon}$ rencontre $\mathbb{C}_{t_{\ast}}^{2}$ sur $\varphi\left(
b\Delta_{d}\right)  $, on remplace cette cellule par un courant de la forme
$\widetilde{\varphi}_{\ast}\left[  \Delta_{d}\right]  -d\theta$ o\`{u}, $\tau$
\'{e}tant un r\'{e}el strictement positif assez petit, $\widetilde{\varphi
}=\varphi\circ\tau Id_{\Delta_{d}\ast}$ et $\theta$ est le courant
d'int\'{e}gration naturel sur $\left\{  \alpha x~;\left(  \alpha,x\right)
\in\left[  1,\tau\right]  \times\Delta_{d}\right\}  $. En effectuant cette
op\'{e}ration sur chaque cellule de $\mathcal{P}_{\varepsilon}$ telle que
$\varphi\left(  b\Delta_{d}\right)  \cap\mathbb{C}_{t_{\ast}}^{2}%
\neq\varnothing$, on obtient, si $\tau$ est assez petit et bien choisi, une
d\'{e}composition $T=%
P%
_{\varepsilon}^{t_{\ast}}+d\Theta_{\varepsilon}^{t_{\ast}}$ adapt\'{e}e \`{a}
$t_{\ast}$ au sens o\`{u} outre les propri\'{e}t\'{e}s de la d\'{e}composition
pr\'{e}c\'{e}dente, $\left(
P%
_{\varepsilon}^{t_{\ast}},\Theta_{\varepsilon}^{t_{\ast}}\right)  $ jouit
aussi de celle que $\mathbb{C}_{t_{\ast}}^{2}$ ne coupe les cellules de $%
P%
_{\varepsilon}^{t_{\ast}}$ que sur leurs faces et transversalement. Dans une
telle d\'{e}composition, $\left(
P%
_{\varepsilon}^{t_{\ast}}\right)  _{t}$ est bien d\'{e}fini pour $t$ voisin de
$t_{\ast}$ et il r\'{e}sulte \`{a} nouveau de \cite[th. 4.3]{HaR1972} que pour
$t$ voisin de $t_{\ast}$ les sections $\left(  \Theta_{\varepsilon}^{t_{\ast}%
}\right)  _{t}$ sont bien d\'{e}finies.

Soient $\rho\in C_{c}^{\infty}\left(  \mathbb{R}^{n},\mathbb{R}_{+}\right)  $
d'int\'{e}grale $1$ et $\left(  \rho_{\alpha}\right)  =\left(  \mathbb{R}%
^{n}\ni\theta\mapsto\alpha^{-n}\rho\left(  \theta/\alpha\right)  \right)
_{\alpha>0}$. On fixe $t$ suffisamment voisin de $t_{\ast}$ afin que les
consid\'{e}rations pr\'{e}c\'{e}dentes s'appliquent. On a donc
\[
\left\langle T,dh\wedge P_{t}\right\rangle =\,\underset{\alpha\downarrow
0}{\lim}\left\langle T,dh\wedge P_{t}^{\alpha}\right\rangle
\]
o\`{u} pour $\alpha>0$, $P_{t}^{\alpha}$ est la forme obtenue en convolant les
coefficients de $P_{t}$ avec $\rho_{\alpha}$. Si $\alpha>0$, $\left\langle
T,dh\wedge P_{t}^{\alpha}\right\rangle =\left\langle T,h\wedge dP_{t}^{\alpha
}\right\rangle $ (car $T$ est ferm\'{e}) et $dP_{t}^{\alpha}=\rho_{\alpha
}\left(  .-t\right)  d\theta$. D'apr\`{e}s la d\'{e}finition des sections de
$T$, on en d\'{e}duit en faisant tendre $\alpha$ vers $0^{+}$ que
\[
\left\langle T,dh\wedge P_{t}\right\rangle =\left\langle T_{t},h\right\rangle
=I_{T,h}\left(  t\right)
\]
De m\^{e}me, $\left\langle
P%
_{\varepsilon}^{t_{\ast}},dh\wedge P_{t}\right\rangle =I_{%
P%
_{\varepsilon}^{t_{\ast}},h}\left(  t\right)  $ et $\left\langle
d\Theta_{\varepsilon}^{t_{\ast}},dh\wedge P_{t}\right\rangle =\left\langle
\left(  \Theta_{\varepsilon}^{t_{\ast}}\right)  _{t},dh\right\rangle =0$, la
derni\`{e}re \'{e}galit\'{e} \'{e}tant due au fait que $dh$ est une $\left(
2,0\right)  $-forme diff\'{e}rentielle de $\mathbb{C}_{t}^{2}$ et que le
support de $\left(  \Theta_{\varepsilon}^{t_{\ast}}\right)  _{t}$ est contenu
dans $N_{t}$ qui est une courbe complexe. D'o\`{u}
\[
I_{T,h}\left(  t\right)  =\left\langle T,dh\wedge P_{t}\right\rangle
=\left\langle
P%
_{\varepsilon}^{t_{\ast}},dh\wedge P_{t}\right\rangle +\left\langle
d\Theta_{\varepsilon}^{t_{\ast}},dh\wedge P_{t}\right\rangle =I_{%
P%
_{\varepsilon}^{t_{\ast}},h}\left(  t\right)  .
\]
La d\'{e}composition $T=%
P%
_{\varepsilon}^{t_{\ast}}+d\Theta_{\varepsilon}^{t_{\ast}}$ \'{e}tant
adapt\'{e}e \`{a} $t_{\ast}$, $I_{%
P%
_{\varepsilon}^{t_{\ast}},h}$ et donc $I_{T,h}$ est r\'{e}elle analytique au
voisinage de $t_{\ast}$.\medskip

Finalement, $I_{T,h}$ est analytique sur $\mathbb{R}^{n}$ et donc nulle
puisque nulle \`{a} l'infini. Les sections $T_{t}$ de $T$ bordent donc des
1-chaines holomorphes $X_{t}$ de $\mathbb{C}_{t}^{2}\backslash
\operatorname*{Spt}T_{t}$. De m\^{e}me, les sections $\left(
P%
_{\varepsilon}^{t_{\ast}}\right)  _{t}$ et $%
P%
_{\varepsilon}^{t_{\ast}}$ sont aussi les bords de 1-chaines holomorphes
$Y_{\varepsilon,t}^{t_{\ast}} $ de $\mathbb{C}_{t}^{2}\backslash
\operatorname*{Spt}T_{t}$. Par construction, $X_{t}$ et $Y_{\varepsilon
,t}^{t_{\ast}}$ co\"{\i}ncident en dehors d'un voisinage (ferm\'{e}) d'ordre
$2\varepsilon$ de $\operatorname{Supp}T_{t}$ et si $t$ est suffisamment voisin
de $t_{\ast}$, $X_{t}=\,\underset{\varepsilon\downarrow0}{\lim}\left(
Y_{\varepsilon}^{t_{\ast}}\right)  _{t}$.

Comme dans la section pr\'{e}c\'{e}dente, on d\'{e}finit ce qui sera le
support de la cha\^{\i}ne CR cherch\'{e}e par ses sections. Si $t\in
\mathbb{R}^{n}$ on note $\mathcal{X}_{t}$ le support de la 1-cha\^{\i}ne
$X_{t}$ et on pose
\[
\mathcal{X}=\underset{t\in\mathbb{R}^{n}}{\cup}\mathcal{X}_{t}\text{~}.
\]
Hormis l'utilisation du lemme~\ref{L/ chgt var} qui permet de se ramener au
cas des sous-vari\'{e}t\'{e}s r\'{e}elles analytiques coup\'{e}es
transversalement par $\mathbb{C}_{t_{\ast}}^{2}$ et qui ici est remplac\'{e}e
par la construction de $%
P%
_{\varepsilon}^{t_{\ast}}$, la preuve du lemme~\ref{L/ reg de Xtilde}
s'applique et on obtient que $\mathcal{X}$ est un sous-ensemble analytique de
$\mathbb{E}^{n,2}\backslash M$ de dimension $n+2$, que ses sections sont des
courbes complexes et que la trace de son ensemble singulier $\mathcal{X}^{s}$
sur ces sections sont des ensembles discrets et $\mathcal{H}^{n}\left(
\mathcal{X}^{s}\right)  <+\infty$.

Prouvons maintenant que$\mathcal{X}$ est le support d'un courant
d'int\'{e}gration $X$ de $\mathbb{E}^{n,2}\backslash M$ v\'{e}rifiant $dX=T$
et $dX_{t}=T_{t}$ pour tout $t\in\mathbb{R}^{n}$. Soit $p_{\ast}=\left(
\zeta_{\ast},w_{\ast},t_{\ast}\right)  $ un point de la partie
r\'{e}guli\`{e}re $\mathcal{X}^{r}$ de $\mathcal{X}$. Ayant fix\'{e}
$\varepsilon$ dans $\left]  0,\frac{1}{2}dist\left(  p_{\ast},M\right)
\right[  $ et not\'{e} $\mathcal{X}^{\varepsilon}$ le support du courant $%
P%
_{\varepsilon}^{t_{\ast}}$ pr\'{e}c\'{e}demment d\'{e}fini, on utilise les
notations de la preuve du lemme~\ref{L/ reg de Xtilde}~pour $\mathcal{X}%
^{\varepsilon}$ : pour un voisinage $G$ de $p_{\ast}$ et $V=\pi\left(
G\right)  $,
\[
\mathcal{X}^{\varepsilon}\cap G=\left\{  \left(  t,\zeta,w\right)  \in
G~;~\left(  \zeta,t\right)  \in V~et~(PQ)\left(  \zeta,t,w\right)  =0\right\}
\]
o\`{u} $P$ et $Q$ sont premiers entre eux dans $CR^{\omega}\left(  V\right)
\left[  w\right]  $. Etant donn\'{e} que $\mathcal{X}^{\varepsilon}$ et
$\mathcal{X}$ co\"{\i}ncident sur $\left\{  dist\left(  .,M\right)
>2\varepsilon\right\}  $ et que $p_{\ast}$ est un point r\'{e}gulier de
$\mathcal{X}^{\varepsilon}$, $\mathcal{X}^{\varepsilon}\cap G$ est, quitte
\`{a} diminuer $G$, le graphe d'une application $F$ de $CR^{\omega}\left(
V\right)  $. Puisque $P$ et $Q$ sont premiers entre eux dans $CR^{\omega
}\left(  V\right)  \left[  w\right]  $ et puisque $p_{\ast}$ est un point
r\'{e}gulier de $\mathcal{X}^{\varepsilon}$, on en d\'{e}duit par un argument
classique de formule de r\'{e}sidu logarithmique que $R\left(  \zeta
,t,w\right)  =\frac{P\left(  \zeta,t,w\right)  }{Q\left(  \zeta,t,w\right)  }$
se factorise dans $G$ sous la forme $\left[  w-F\left(  \zeta,t\right)
\right]  ^{\mu\left(  p_{\ast}\right)  }S\left(  \zeta,t,w\right)  $ o\`{u}
$S\in CR^{\omega}\left(  V\right)  \left(  w\right)  $ n'a ni p\^{o}le ni
z\'{e}ro dans $G$ et $\mu\left(  p_{\ast}\right)  =\pm m\left(  p_{\ast
}\right)  $ selon que $w_{\ast}$ est un z\'{e}ro ou un p\^{o}le de $R\left(
\zeta_{\ast},t_{\ast},.\right)  $. Si $t$ est suffisamment voisin de $t_{\ast
}$ et si $U$ un voisinage suffisamment petit de $\left(  \zeta_{\ast},w_{\ast
}\right)  $, on a donc%
\begin{equation}
X_{t}\left\vert _{U}\right.  =\mu\left(  p_{\ast}\right)  \left[
\mathcal{X}_{t}\right]  \left\vert _{U}\right.  .
\label{F/ X avec mult et int}%
\end{equation}

L'application $\mathcal{X}^{r}\overset{\mu}{\rightarrow}\mathbb{Z}^{\ast}$
ainsi construite est localement constante. Notons $\left(  \mathcal{X}%
_{\alpha}\right)  _{\alpha\in A}$ la famille (forc\'{e}ment finie) des
composantes connexes de $\mathcal{X}$. Puisque $\mathcal{X}^{s}$ ne
d\'{e}connecte aucune composante de $\mathcal{X}^{r}$, $\mu$ prend une valeur
constante $\mu_{\alpha}$ sur chaque $\mathcal{X}_{\alpha}\cap\mathcal{X}^{r}$
et pour tout $t\in\mathbb{R}^{n}$on obtient%

\begin{equation}
X_{t}=\sum\limits_{\alpha\in A_{t}}\mu_{\alpha}\left[  \mathcal{X}_{\alpha
,t}\right]  . \label{F/ X avec mult et int bis}%
\end{equation}
o\`{u} $A_{t}$ est une partie de $A$. Orientons chaque $\mathcal{X}_{\alpha}$
par $i\left(  dz_{1}\wedge d\overline{z_{1}}+dz_{2}\wedge d\overline{z_{2}%
}\right)  \wedge dt$ de sorte que $\left[  \mathcal{X}_{\alpha}\right]
_{t}=\left[  \mathcal{X}_{\alpha,t}\right]  $~\ et posons%
\[
X=\sum\limits_{\alpha\in A}\mu_{\alpha}\left[  \mathcal{X}_{\alpha}\right]  .
\]
Puisque $\left[  \mathcal{X}_{a}\right]  $ est un courant localement
rectifiable de masse finie, son bord est localement un multiple de $\left[
M\backslash S\right]  $ (voir \cite{FeH1969Li}) et il en est donc de m\^{e}me
pour $dX$. On a donc $\left\langle dX,\Phi\right\rangle =0$ si $\Phi\in
C_{n+1}^{\infty}\left(  \mathbb{E}^{n,2}\right)  $ a un support qui ne
rencontre pas l'ensemble $M^{\prime}$ des points $p=\left(  t,z\right)  $
o\`{u} $M$ et $\mathbb{C}_{t}^{2}$ se coupent et sont transverses~; notons que
dans ce cas on a aussi $\left\langle T,\Phi\right\rangle =0$ puisque tout
$\left(  n+1\right)  $-champ de vecteurs qui oriente $M\backslash S$ est
$\mathcal{H}^{n+1}$-presque partout factorisable par $\underset{1\leqslant
i\leqslant n}{\wedge}\partial/\partial t_{j}$. Si par contre $\Phi$ est de la
forme $dt\wedge\varphi$, alors d'apr\`{e}s \cite{FeH1969Li}%
\begin{align*}
\left\langle dX,\Phi\right\rangle  &  =-\sum\limits_{\alpha\in A}\mu_{\alpha
}\left\langle \left[  \mathcal{X}_{\alpha}\right]  ,dt\wedge d_{z}%
\varphi\right\rangle =-\sum\limits_{\alpha\in A}\mu_{\alpha}\int
_{t\in\mathbb{R}^{n}}\left\langle \left[  \mathcal{X}_{\alpha}\right]
_{t},d_{z}\varphi\right\rangle d\mathcal{H}^{n}\left(  t\right) \\
&  =\int_{t\in\mathbb{R}^{n}}\left\langle d_{z}\sum\limits_{\alpha\in A}%
\mu_{\alpha}\left[  \mathcal{X}_{\alpha,t}\right]  ,\varphi\right\rangle
d\mathcal{H}^{n}\left(  t\right) \\
&  =\int_{t\in\mathbb{R}^{n}}\left\langle T_{t},\varphi\right\rangle
d\mathcal{H}^{n}\left(  t\right)  =\left\langle T,dt\wedge\varphi\right\rangle
=\left\langle T,\Phi\right\rangle .
\end{align*}
Comme $\mathcal{H}^{n+1}\left(  M\backslash M^{\prime}\right)  =0$, on en
d\'{e}duit que la formule $dX=T$ est valable au sens des courants. Le fait que
$dX_{t}=T_{t}$ pour tout $t\in\mathbb{R}^{n}$ r\'{e}sulte de
(\ref{F/ X avec mult et int bis}). Lorsque $M\backslash S$ est connexe, on
obtient que $X\in\mathbb{Z}\left[  \mathcal{X}\right]  $ de la m\^{e}me
fa\c{c}on que dans~\cite[th 3.3 \& p. 370]{HaR1977}.

Lorsque $p=\left(  t,z\right)  $ est en outre un point r\'{e}gulier de $M$
o\`{u} $M$ est transverse \`{a} $\mathbb{C}_{t}^{2}$, on sait d'apr\`{e}s
\cite[th. II]{HaR-LaB1975} que si $z$ poss\`{e}de dans $\mathbb{C}_{t}^{2}$ un
voisinage $U_{p}$ tel que chaque composante connexe de $\overline
{\mathcal{X}_{t}}\cap U_{p}$ est soit une vari\'{e}t\'{e} \`{a} bord
r\'{e}elle analytique de bord $M_{t}\cap U_{p}$, soit un sous-ensemble
r\'{e}el analytique de $U_{p}$~; en utilisant aussi \cite[th. 4.7]%
{HaR-LaB1975}, on obtient que si $%
C%
$ est une composante connexe de $\mathcal{X}_{t}\cap U_{p}$, $U_{p}\cap b%
C%
=M\cap U_{p}$ quitte \`{a} diminuer $U_{p}$. Par transversalit\'{e}, cette
conclusion s'\'{e}tend \`{a} $\left(  \mathcal{X},M,p\right)  $ dans $W_{p}$
si $W_{p}$ est suffisamment petit. Si $p$ n'est pas un point r\'{e}gulier de
$M$ o\`{u} $M$ est transverse \`{a} $\mathbb{C}_{t}^{2}$, ces arguments
s'applique malgr\'{e} tout \`{a} $%
P%
_{\varepsilon}^{t}$ pour des r\'{e}els $\varepsilon$ arbitrairement petit. Il
est clair que ceci entra\^{\i}ne que si $W$ est un voisinage assez petit de
$p$, $M$ contient au moins l'une des composantes connexes de $(W\cap
N)\backslash M$. Ceci ach\`{e}ve la preuve du
th\'{e}or\`{e}me~\ref{T/bord CxR B}.

\renewcommand\refname{R\'ef\'erences}
\renewcommand\baselinestretch{1}{\normalsize
\bibliographystyle{amsperso}
\bibliography{ref}
}%

\end{spacing}%

\end{document}